\documentclass[a4paper]{article}
\usepackage{geometry}
\usepackage{amssymb}
\usepackage{latexsym}
\usepackage{amsmath}
\usepackage{indentfirst}
\usepackage{graphicx}
 \usepackage[colorlinks=true]{hyperref}
\usepackage{mathrsfs} 
\usepackage{chngcntr}
\usepackage{color}
\usepackage [latin1]{inputenc}

\newtheorem{theorem}{Theorem}[section]
\newtheorem{Cor}{Corollary}[section]
\newtheorem{lemma}{Lemma}[section]

\newtheorem{remark}{Remark}[section]
\newtheorem{Def}{Definition}[section]

\newcommand{\F}{\mathcal{F}}
\newcommand{\D}{\mathcal{D}}

\numberwithin{equation}{section}

\newenvironment{proof}{\medskip\par\noindent{\bf Proof\/}:\quad}{\qquad
\raisebox{-0.5mm}{\rule{1.5mm}{1mm}}\vspace{6pt}}

\begin{document}
\title{Infinitely many  normalized solutions for a  quasilinear {S}chr\"{o}dinger  equation}
 \author{Xianyong Yang$^1$\, \quad Fukun Zhao$^2$ \thanks{Corresponding author. E-mail addresses:fukunzhao@163.com}\\
\footnotesize{\em 1. School\, of\,  Mathematics and Computer Science,\,Yunnan\,
Minzu\, University, \,Kunming \,650500\, P.R. China }\\
\footnotesize{\em 2. School\, of\,  Mathematics, \,Yunnan\,
Normal\, University, \,Kunming \,650500\, P.R. China }
}
\date{}
\maketitle

\begin{abstract}
In this paper, we are concerned with the following quasilinear {S}chr\"{o}dinger equation
\begin{equation*}
\begin{cases}
-\Delta u+ \mu u-\Delta(u^2)
u=g(u)~~\hbox{in}~~\mathbb{R}^N\\
\int_{\mathbb{R}^N}|u|^2dx=m,\\
u\in H^1(\mathbb{R}^N),
\end{cases}
\end{equation*}
where $N\geq2$, $m>0$ is a given constant, $\mu\in \mathbb{R}$ is a Lagrange multiplier, and $g$ satisfies the well-known Berestycki--Lions condition. The existence of infinitely many  normalized solutions is obtained via a minimax argument.
\end{abstract}

\section{Introduction and  main results}
Assume $N\geq2$ and consider the following quasilinear Schr\"{o}dinger equation
\begin{equation}\label{eqA1}
-\Delta u+\lambda u-  \Delta(u^2)u=g(u)~~\hbox{in}~~\mathbb{R}^N,
\end{equation}
which is related to the standing waves $\psi(x,t)=u(x)e^{-iEt}$ of the following
Schr\"{o}dinger equation

 \begin{equation}\label{eqA3}
-i\psi_t-\Delta\psi+\mu\psi-h(|\psi|^2)\psi-\Delta(|\psi|^2)\psi=0,
\end{equation}
 where  $i$ denotes the imaginary unit, $\lambda=\mu-E$, $g(t)=h(|t|^2)t$. The quasilinear Schr\"{o}dinger equation \eqref{eqA3} is derived as a model of several physical phenomena. For example, it was used for the superfluid film equation in plasma physics by  Kurihara
\cite{Kurihara1981Large}. It also appears in the theory of
Heisenberg ferromagnetism and magnons (see \cite{Kosevich1990Magnetic} and \cite{MR647411}), in dissipative quantum mechanics and the condensed matter theory \cite{Makhankov1984}.
In literature the study of standing waves of \eqref{eqA1} has been pursed in two main directions, which opened two different challenging research fields.
\par
 The first topic regards the search for solutions of \eqref{eqA1} with a prescribed frequency $\mu$ and free mass, which is called unconstrained problem.
To this aim,  we consider the associated energy functional to \eqref{eqA1}
\begin{equation*}
I(u)={1\over 2}\int_{\mathbb{R}^N}\big(1+2 u^2\big)|\nabla
u|^2dx+\lambda \int_{\mathbb{R}^N}u^2dx-\int_{\mathbb{R}^N}G(u)dx
\end{equation*}
defined on the natural space
\begin{equation*} X=\bigg\{u\in H^1(\mathbb{R}^N)|\int_{\mathbb{R}^N}u^2|\nabla
u|^2dx<\infty\bigg\},\end{equation*}
where $G(u)=\int_0^ug(s)ds.$
A function $u\in X$ is called  a weak solution of \eqref{eqA1} if
$$\lim_{t\rightarrow0}\frac{I(u+t\varphi)-I(u)}{t}=0$$
for all $\varphi \in C_0^\infty(\mathbb{R}^N).$ Recall that $22^\ast=\frac{4N}{N-2}$ corresponds to a critical exponent of \eqref{eqA1} when $g(u)=|u|^{r-2}u$ and $N\geq3$, see \cite{Liu-Wang-Wang2003JDE}. Different from the semilinear case that $\kappa=0$, there is a non-convex term $\Delta(u^2)u$ appears in the quasilinear problem \eqref{eqA1}, and hence the search of solutions of such an equation must face a challenge: the functional
 $$W(u)=\int_{\mathbb{R}^N}u^2|\nabla
u|^2dx$$
associated with the quasilinear term
is non differentiable in the space $X$ when $N\geq2$. Consequently,  the standard critical point theory can not be applied directly. In order to overcome this difficulty, several approaches have been successfully developed in the last decade, such as the constraint minimization (\cite{LiuWang2003}), the Nehari manifold
method (\cite{Liu-Wang-Wang2004CPDE}), the perturbation method (\cite{LiuLiuWang2013}) and the nonsmooth critical point theory (\cite{Canino-Degiovanni1995,Lizhouxin2016AMSS}). It seems that  the first contribution to the  equation  \eqref{eqA1} via variational methods dues to the work \cite{CMP1997}. After this work, a series of subsequent studies have been done concerning with the  existence, multiplicity and  qualitative properties of solutions to \eqref{eqA1}, see for example \cite{Pucci2018DCDS, doO-Miyagaki2010JDE, LZL2022Nonlinearity, LiuLiuWang2014CPDE, LiuWang2014JDE,  Liu-Wang-Guo2012JFA, Liu2013CVPDE, SchmittWang2002CVPDE}.

The second topic of investigation of the problem \eqref{eqA1} consists of prescribing the mass $m>0$, thus conserved by $\psi$ in time
$$\int_{\mathbb{R}^N}\psi(x,t)dx=m,\forall t\in[0,+\infty),$$ and letting the frequency $\mu$ to be free. Such a problem is usually called constrained one.
The constrained problem has a significant relevance in Physics, not only for the quantum probability normalization, but also because the mass may also have specific meanings, such as the power supply in nonlinear optics, or the total number of atoms in Bose-Einstein condensation. Moreover, the investigation of constrained problems can give a better insight of the dynamical properties, as the orbital stability of solutions of \eqref{eqA3}, see for example \cite{CingolaniTanaka2022CVPDE}. For the case $\kappa=0$, \eqref{eqA1} is reduced to the semilinear {S}chr\"{o}dinger equation whose normalized solutions are studied widely in recent years. We can  refer  the interested readers to \cite{CazenaveCMP1982,CingolaniTanaka2022CVPDE, HajaiejANS2004,JeanjeanNonlinearity2019, Shibata143MM} and the references therein.
\par
In recent years, there has been an increasing interest in studying of normalized solutions  to the quasilinear Schr\"{o}dinger equation.
  To obtain normalized solutions, a natural method is to look for the minimizer of the following constrained problem:
\begin{equation*}
e(m):=\inf_{u\in S(m)}E(u),
\end{equation*}
where
\begin{equation*} \label{eq1.5}S(m)=\bigg\{u\in H^1(\mathbb{R}^N):\int_{\mathbb{R}^N}u^2|\nabla
u|^2dx<\infty~\text{and}~\int_{\mathbb{R}^N}|u|^2dx=m\bigg\},\end{equation*}
and
\begin{equation*}
E(u)={1\over 2}\int_{\mathbb{R}^N}\big(1+2u^2\big)|\nabla
u|^2dx-\int_{\mathbb{R}^N}G(u)dx.
\end{equation*}
It is proved in\cite[Theorem 4.6]{JeanjeanNonlinearity2010} that each minimizer of $e(m)$, corresponds to a Lagrange
multiplier $\lambda<0$ such that $(u,\lambda)$ solves weakly $\eqref{eqA1}$ when $g(u)=|u|^{r-2}u$.
Moreover, it is easy to see that  $e(m)>-\infty$ for $2<r<4+\frac{4}{N}$ and  $e(m)=-\infty$ for $4+\frac{4}{N}<r<\frac{4N}{N-2}$.
From this point of view,  $r=4+\frac{4}{N}$ plays a  critical exponent for the constrained problem.
By this constrained  argument, Colin, Jeanjean and Squassina \cite{JeanjeanNonlinearity2010} first prove the following result for $L_2$ -constrained problem (see also \cite{JeanjeanZAMP2013}).
 \begin{theorem}\label{theorem1}
Assume that $g(u)=|u|^{r-2}u$, where $r\in \big(2,\frac{4N}{N-2}\big)$ if $N\geq3$ and $r\in (2,\infty)$ if $N=1,2$. Then
\begin{itemize}
\item[$(1)$] For all $m>0$ and $r\in(2,2+\frac{4}{N})$, $e(m)\in(-\infty,0)$ and $e(m)$ has a minimizer.
\item[$(2)$] For $r\in(2+\frac{4}{N},4+\frac{4}{N})$, there exists $m(r,N)>0$ such that
\subitem(\expandafter{\romannumeral1})If $m\in\big(0,m(r,N)\big), e(m)=0$ and $e(m)$ has  no minimizer;
\subitem(\expandafter{\romannumeral2})If $m=m(r,N), m(c)=0$ and $e(m)$ has  a minimizer;
\subitem(\expandafter{\romannumeral3}) If $m\in\big(c(r,N),+\infty\big), e(m)<0$ and $e(m)$ has  a minimizer;
\item[$(3)$] For all $m>0$, $r\in \big(4+\frac{4}{N},\frac{4N}{N-2}\big)$ if $N\geq3$ and $r\in \big(4+\frac{4}{N},\infty)$ if $N=1,2$, there holds $e(m)=-\infty$.
\item[$(4)$] For $r=4+\frac{4}{N}$,   there exists $m_N>0$ such that
 \begin{equation*}
e(m)=
\begin{cases}
0,~~~~~~\hbox{if}~m\in(0,m_N),\\
-\infty,~~\hbox{if}~m\in(m_N,\infty).
\end{cases}
\end{equation*}
Moreover, $e(m)$ has no minimizer  for all $m>0$.
\item[$(5)$] The standing waves obtained as minimizer of $ e(m)$ are orbitally stable.
\end{itemize}
\end{theorem}

 Interestingly,    by a combination of analytical and numerical arguments in dimension
$N = 3$,  Caliari and Squassina \cite{CaliariJFPTA2012} studied the  explicit bounds on $m(r,N)$ and $m_N$ given in Theorem \ref{theorem1}.  In particular, for $r=\frac{13}{3}$, they find that there exists $\underline{m},~\overline{m}>0$ such that
 \begin{equation*}
 e(m)=
\begin{cases}
0,~~~~~~\hbox{if}~m\in(0,\underline{m}),\\
-\infty,~~\hbox{if}~m\in(\overline{m},\infty).
\end{cases}
\end{equation*}
 Moreover,  $\underline{m} \approx 19.73$ and $\overline{m} \approx 85.09$.
  However, the constrained minimization  argument does not use much of smoothness of the variational functional,  so it is not suitable for dealing with the existence of multiple solutions. Subsequently, Jeanjean, Luo and Wang  \cite{WZQ2015JDE}  prove the existence of two normalized solutions for $L_2$ -subcritical case   via the perturbation approach. More precisely, they proved the following  result.
  \begin{theorem}\label{theorem2}
 Assume that  $N\geq1$ and $g(u)=|u|^{r-2}u$ with $r\in(2+\frac{4}{N},4+\frac{4}{N})$.
 \begin{itemize}
 \item[$(1)$] There exists a $m_0\in(0,m(r,N))$ such that for any $m\in (0,m(r,N))$ the functional $E$ admits a
critical point $v_m$ on $S(m)$ which is a local minimum of $E$ when $m\in (m_0,m(r,N))$ and a global minimum of $E$ when $m\in (m(r,N),+\infty)$. In particular,
 \begin{equation*}
 E(v_m)
\begin{cases}
>0,~~\hbox{if}~m\in(m_0,m(r,N));\\
=0,~~\hbox{if}~m=m(r,N);\\
<0,~~\hbox{if}~m=(m(r,N),+\infty).
\end{cases}
\end{equation*}
\item[$(2)$]  Assuming in addition that $r\in(2+\frac{4}{N},2+\frac{4}{N-2})$ if $N\geq5$ there exists a second critical point $u_m$ on $S(m)$
 which satisfies
\subitem(\expandafter{\romannumeral1}) $E(u_m)>0$ for all $m\in(m_0,+\infty)$ and is a mountain pass level;
\subitem(\expandafter{\romannumeral2})$E(u_m)>E(v_m)$ for any $m\in(m_0,+\infty)$.
\end{itemize}
\end{theorem}

Very recently, Ye and Yu \cite{YY2021JMAA} study the $L_2$-critical problem to \eqref{eqA1} if  $g(u)=|u|^{r-2}r$ for $r=2+\frac{4}{N}$ and $r=4+\frac{4}{N}$. They prove the following result.
\begin{theorem}\label{Thm1.3}
Assume that $g(u)=|u|^{r-2}u$.  Then
\begin{itemize}
\item[$(1)$] For $r=2+\frac{4}{N}$, there exists $m_N>0$ such that
\subitem(\expandafter{\romannumeral1})If $m\in\big(0,m_N\big], e(m)=0$ and $E(u)$ has no critical point on the constraint $S(m)$;
\subitem(\expandafter{\romannumeral2}) If $m\in\big(m_N,+\infty\big), e(m)<0$ and $e(m)$ has  a minimizer;
\end{itemize}
\end{theorem}\par
\begin{itemize}
\item[$(2)$] For $r=4+\frac{4}{N}$, there exists $m_N>0$ such that
\subitem(\expandafter{\romannumeral1}) $E(u)$ has  no critical point on the constraint $S(m)$ for all $m\in\big(0,m_N\big]$;
\subitem(\expandafter{\romannumeral2}) $E(u)$ has a critical point on the constraint $S(m)$ for all  $m\in\big(m_N,+\infty\big)$ when $N\leq3$.
\end{itemize}
\par
We also point out that  some works focus on the existence and  asymptotic behavior of  normalized solutions related to some quasilinear elliptic equations with  potential function, that is, $\mu$ is replaced by $V(x)$  satisfying different assumptions in \eqref{eqA1}, we can  refer  the interested readers to \cite{Phys.D2019,Zou2023PJM,WZQ2023,ZZ2018ANS,ZZ2019DCDS} and the references therein for this case.

 However, to the best of our knowledge, the existing results in references of the constrained problem \eqref{eqA1} mainly focus on the case where $g$ is a power function. It seems that there is no work considering the existence of normalized solutions for \eqref{eqA1} when $g$ is a general nonlinearity. A natural question is whether the nonlinearity $g(u)=|u|^{r-2}u$ can be relaxed to a general nonlinearity $g$ or not? On the other hand, it seems that the multiplicity of normalized solutions for \eqref{eqA1} only can be found in Jeanjean, Luo and Wang  \cite{WZQ2015JDE}, where the authors partly obtained the existence of two normalized solutions for \eqref{eqA1} when  $g(u)=|u|^{r-2}u$. Since such a power nonlinearity is odd, another  natural question is whether \eqref{eqA1} has infinitely many normalized solutions or not? Moreover, we note from Theorem \ref{theorem2} that  $r\in \big(2+\frac{4}{N},4+\frac{4}{N}\big)$ for $N\leq4$ or $r\in \big(2+\frac{4}{N},2+\frac{4}{N-2})$ for $r\geq5$ are required. However, it is well-known that the existence  of  solutions of \eqref{eqA1} can be obtained  for  the unconstrained problem when $r\in \big(2,\frac{4N}{N-2}\big)$ and $N\geq1$. So the last natural questions is whether the scope of $r$ and the dimension of the space can be relaxed to obtain multiple normalized solutions of \eqref{eqA1}?
\par
Motivated by the above results especially by \cite{JeanjeanZAMP2013} and  \cite{WZQ2015JDE}, in the present paper, we are concerned with the existence and multiplicity of normalized solutions of the following quasilinear {S}chr\"{o}dinger  equation
\begin{equation*}
\begin{cases}
-\Delta u+ \mu u-\Delta(u^2)u=g(u)~~\hbox{in}~~\mathbb{R}^N,\\
\int_{\mathbb{R}^N}|u|^2dx=m,~~~~~~~~~~~~~~~~~~~~~~~~~~~~~~~~~~~~~~~~~~~~~~~~~~~~~~~~~~~~~~~~~~~~~~~~~~~~~(P_{\mu,m})\\
u\in H^1(\mathbb{R}^N),
\end{cases}
\end{equation*}
where $g$ satisfies the well-known Berestycki--Lions condition, which is regarded as an almost optimal assumption of the existence of solutions for nonlinear Schr\"{o}dinger equation (see \cite{BL19831,BL19832}).  To state our main results, we need the following assumptions.
 \begin{itemize}
\item[$(g_1)$] $g\in C(\mathbb{R},\mathbb{R})$.
\item[$(g_2)$] $\lim\limits_{s\rightarrow0}\frac{g(s)}{s}=0$.
\item[$(g_3)$] For $p=4+\frac{4}{N}$, $\lim\limits_{s\rightarrow\infty}\frac{|g(s)|}{|s|^{p-1}}=0$.
\item[$(g_4)$] There exists some $s_0>0$ such that $G(s_0)>0$.
\item[$(g_5)$] $g(-s)=-g(s)$ for all $s\in \mathbb{R}$.
\item[$(g_6)$] For $q=2+\frac{4}{N}$,  $\liminf\limits_{s\rightarrow0}\frac{|g(s)|}{|s|^{q-2}s}=\infty$.
\item[$(g_7)$] $g(s)s\geq0$ for all $s\in\mathbb{R}$.
\end{itemize}

\begin{Def} For any $m > 0$, a solution $u\in H^1(\mathbb{R}^N)$ to \eqref{eqA1} is called a least energy normalized solution, if $u\in S(m)$ and
$$E(u) = \min\{E(u) : u \in S(m),\mu\in\mathbb{R} ~\text{and} ~(u,\mu)~ \text{solves} ~(P_{\mu,m})\}.$$\end{Def}
\par
Now, our main results of this paper can be stated as follows.
 \begin{theorem}\label{Thm1.10}
Assume that $(g_1)-(g_4)$ hold  and $N\geq2$. Then
 \begin{itemize}
 \item[$(1)$]  There exists $m_1\geq0$  such that $(P_{\mu,m})$ has at least a ground state normalized  solution for  $m>m_1$.
\item[$(2)$] In addition that $(g_6)$,  $(P_{\mu,m})$ has at least a ground state normalized  solution  for any $m>0$.
\end{itemize}
\end{theorem}\par
\begin{theorem}\label{Thm1.1}
Assume that $(g_1)-(g_5)$ hold  and $N\geq2$. Then
 \begin{itemize}
 \item[$(1)$]  For any $k\in \mathbb{N}$, there exists $m_k\geq0$  such that $(P_{\mu,m})$ has at least $k$ radial solutions for $m>m_k$.
\item[$(2)$] In addition that $(g_6)$,  $(P_{\mu,m})$ has infinitely many radial solutions  for any  $m>0$.
\end{itemize}
\end{theorem}\par
We immediately conclude the following result from Theorem \ref{Thm1.1}:
\begin{Cor}\label{Cor1}
  Assume that  $g(u)=|u|^{r-2}u$  hold and $N\geq2$.
\begin{itemize}
 \item[$(1)$] $(P_{\mu,m})$ has infinitely many radial solutions  for any $2<r<2+\frac{4}{N}$ and $m>0$.
 \item[$(2)$] For any $k\in \mathbb{N}$, there exists $m_k>0$  such that $(P_{\mu,m})$ has at least $k$ radial solutions for $2+\frac{4}{N}\leq r<4+\frac{4}{N}$ when $m>m_k$.
\end{itemize}
\end{Cor}
\begin{remark} If $e(m)<0$ for  $r\in \big(2,4+\frac{4}{N})$,  by Theorem \ref{theorem1} and Theorem \ref{theorem2}, we have  $e(m)=b_1^m$ and $m_1=m(r,N)$, where $b_1^m$ is given in definition 3.1 below.  In other words, the minimizer of $E|_{S(m)}$  can be characterized by mountain pass level.  Clearly, Corollary \ref{Cor1} improves Theorems \ref{theorem1}--\ref{theorem2}.
 \end{remark}

\begin{theorem}\label{Thm1.6}
Assume that  $(g_1)-(g_5)$, $(g_7)$ hold  and $N\geq2.$ Then
 \begin{itemize}
 \item[$(1)$]   $e(m)$ has  a minimizer  for  $m>m_1$, where $m_1$ is given in  Theorem \ref{Thm1.10}.
\item[$(2)$] In addition that $(g_6)$,    $e(m)$ has  a minimizer for any $m>0$.
\end{itemize}
\end{theorem}\par

\begin{remark}
It is a natural question that if $e(m)$ has  a minimizer  for  $m\in (0,m_1]$. For semilinear case, by proving a strict subadditivity inequality as follows
$$e_{m_1+m_2}<e_{m_1}+e_{m_2},$$
Shibata \cite{Shibata2014ManuscriptaMath} gives a  negative answer. For quasilinear case, however, it is not clear this subadditivity inequality holds, because the functional $E$  is non differentiable in the space $X$ when $N\geq2$.  So this problem still remains open.
 \end{remark}

To prove Theorems \ref{Thm1.10}--\ref{Thm1.1}, we employ a new approach introduced by Hirata and Tanaka  \cite{Tanaka2019ANS}.  One advantage of this method is that it can be suitably
adapted to derive multiplicity results of normalized solutions by using minimax and deformation arguments.  Firstly, we consider  a  augmented  functional $I^m(\lambda,u)$,
which is  not well defined in the normal Sobolev space $\mathbb{R}\times H^1(\mathbb{R}^N)$. Consequently, one of the  difficulties to deal with \eqref{eqA1}  stems from the fact that there is no suitable working space where the energy functional enjoys  both the smoothness and compactness properties.   In order to overcome this difficulty,
inspired by \cite{Liu-Wang-Wang2003JDE}, we make a change of variables and transform the quasilinear problem into a semilinear one.

 Since  $g$ only satisfies the Berestycki--Lions condition,  it impossible for us to verify the boundedness
 of $(PS)$ or $(C_c)$ sequence because of the absence of Ambrosetti-Rabinowize condition  and monotonicity condition.  Inspired by \cite{Tanaka2019ANS}, we use a new deformation argument under a new version of the Palias--Samle condition related to a  Pohozaev functional. Compared with the  unconstrained  case, we find that the corresponding functional  seems to  satisfies the compactness condition only at the negative energy level but not at positive energy level (see Lemma \ref{LemA7}). Consequently, we have to construct a family of negative minimax values $b^k_m$ to prove Theorems \ref{Thm1.10}--\ref{Thm1.1}. To this aim, a key problem is prove that none of $b^k_m$ is equal to $-\infty$ and hence $b^k_m$ is well defined.
 In \cite{Tanaka2019ANS},  Hirata and Tanaka  accomplished  it  by comparing $b^k_m$ and the least  energy of the following classic problem
\begin{equation}\label{eq666}-\Delta u+u=|u|^{r-2}u,~~x\in\mathbb{R}^N,\end{equation} where $2<r<2^\ast.$  This argument heavily relies on  scaling properties of functional. However, it seems to fail for the  quasilinear case because of the nonhomogeneous  properties of  change of variables and nonlinearity. In this paper, we will  borrow from a Pohozaev Mountain developed in \cite{CingolaniTanaka2022CVPDE} to overcome this difficulty. We point out that  a new  property of change of variables  play a important role in proof, see Lemma \ref{LemA1}--(14).
\par
 To look for a minimizer of $e(m)$, our work uses in particular some arguments of Steiner rearrangement (see \cite{Liebbook2001,Lions1984}) and a new scaling of \cite{JeanjeanNonlinearity2010}. We will see that a key point is show that $e(m)<0$. Indeed, for the homogeneous case that $g(u)=|u|^{r-2}u$, one can achieve this  by using a scaling argument. However,
for the nonhomogeneous case, it is not easy to see whether $e(m)<0$ holds or not. In this paper, by using minimax method,  we first find  a least energy normalized  solution with negative energy  to $(P_{\mu,m})$, which concludes that $e(m)<0$. From this perspective, we gives a strategy for finding a minimizer for constraint problems.
\par
The paper is organized as follows. In Section 2, we will focus on  the variational frame and  deformation lemma. In Section 3, we will  construct the
negative minimax level and and analyze their behavior. In Section 4, we will devote to the proofs of main results.

Notation. For the sake of notational simplicity, we omit integral symbols $dx$ without causing confusion.
  Throughout this paper,  $\rightarrow$ and $\rightharpoonup$ denote
the strong convergence and the weak convergence, respectively. $\|\cdot\|$ denotes the norm in $H^1(\mathbb{R}^N)$. $\|(\lambda,u)\|:=\sqrt{\lambda^2+\|u\|^2}$ denotes the norm of  $(\lambda,u)$ in $\mathbb{R}\times H^1(\mathbb{R}^N)$. For $A\subset \mathbb{R}\times H^1(\mathbb{R}^N)$ and $\rho>0$, $N_\rho(A)=\{(\lambda,u):\hbox{dist}((\lambda,u),A)<\rho\}$.
$\|\cdot\|_p$ denotes the norm in  $L^p(\mathbb{R}^N)$ for $1\leq p\leq\infty$.
$\partial D_k=\{x\in \mathbb{R}^k:|x|=1\},~~D_k=\{x\in \mathbb{R}^k:|x|\leq 1\}. $  $\Sigma_\rho =\big\{v\in H^1(\mathbb{R}^N):\int_{\mathbb{R}^N }\big(|\nabla
v|^2+f^2(v)\big)\leq\rho\big\}$
and $\partial\Sigma_\rho =\big\{v\in H^1(\mathbb{R}^N):\int_{\mathbb{R}^N }\big(|\nabla
v|^2+f^2(v)\big)=\rho\big\}.$ $C_0$,  $C$, $C_i$ denote various positive constants whose value may change from line to line but are not essential to the analysis of the proof.
\par
\section{Palais-Smale-Pohozaev condition and  deformation theory}
\subsection{Palais-Smale-Pohozaev condition}
For the sake of notational simplicity, we denote  $\mu=e^\lambda$ . Since we are looking for normalized solutions  to equation $(P_{\mu,m})$, we consider the following energy functional:
\begin{equation*}
I^m(\lambda,u)={1\over 2}\int_{\mathbb{R}^N}\big(1+2u^2\big)|\nabla
u|^2+{ e^\lambda \over2}\bigg(\int_{\mathbb{R}^N}u^2-m\bigg)
-\int_{\mathbb{R}^N}G(u),
\end{equation*}
which is  not well defined in the normal Sobolev space $\mathbb{R}\times H^1(\mathbb{R}^N)$.  In order to overcome this difficulty,
inspired by \cite{Liu-Wang-Wang2003JDE}, we make a change of variables $v=f^{-1}(u)$, where $f$ is defined by
\begin{equation*}
\label {eqD1}
\begin{cases}
f(0)=0,\\
f^\prime(t)=\big(1+2 |f(t)|^2\big)^{-\frac{1}{2}},~~~\forall~t>0,\\
f(t)=-f(-t),~~~~~~~~~~~~~~~~\forall~t<0.
\end{cases}
\end{equation*}
\begin{lemma}\label{LemA1}
The function $f$ satisfies the following properties:
\begin{itemize}
\item[$(1)$] $f\in C^\infty(\mathbb{R},\mathbb{R})$ is uniquely defined  function and invertible.

\item[$(2)$] $0<f^{\prime}(t)\leq 1$ for all $t\in \mathbb{R} $.

\item[$(3)$] $|f(t)|\leq |t|$ for all $t\in \mathbb{R} $.

\item[$(4)$] $\lim\limits_{t\rightarrow0}\frac{f(t)}{t}=1.$

\item[$(5)$] $\lim\limits_{t\rightarrow\infty}\frac{|f(t)|^2}{|t|}=\sqrt{2}.$
\item[$(6)$] $\frac{1}{2}f(t)\leq t f^{\prime}(t)\leq f(t)$ for all $t\geq 0$.

\item[$(7)$] $|f(t)|\leq 2^{\frac{1}{4}}|t|^{\frac{1}{2}}$ for all $t\in \mathbb{R} $.

\item[$(8)$] There exists a positive constant $C>0$ such that
\begin{equation*}
f(t)\geq
\begin{cases}
C |t|, ~~|t|\leq 1,\\
C|t|^{\frac{1}{2}}, |t|> 1.
\end{cases}
\end{equation*}

\item[$(9)$] $|f(t)|f^{\prime}(t)\leq \frac{1}{\sqrt{2}}$ for all $t\in \mathbb{R} $.

\item[$(10)$]  For all $\theta>0$, there is a constant $C(\theta)>0$ such that $|f(\theta t)|^2\leq C(\theta)|f(t)|^2$ for all $t\in \mathbb{R} $.

\item[$(11)$] The function $f(t)f^{\prime}(t)t^{-1}$ is decreasing for
$t>0$.

\item[$(12)$] The function $f^r(t)f^{\prime}(t)t^{-1}$ is increasing for
$r\geq 3$ and $t>0$.

\item[$(13)$]
There exists $C>0$ such that $
|t|^r\leq C|f(t)|^2+C|f(t)|^{2r}$ for all $r\geq2$ and $t \in \mathbb{R}$.

\item[$(14)$]  There exists a positive constant $C>0$ such that  $f^2(rt)\geq Crf^2(t)$  for all $t\in \mathbb{R} $ and $r\geq1$; $f^2(rt)\geq Cr^2f^2(t)$  for all $t\in \mathbb{R} $ and $r\leq1$.
\item[$(15)$]  Assume  $(g_1)$--$(g_3)$  hold. For any $\varepsilon>0$,  there exists $C_\varepsilon>0$ such that  $
|G[f(t)]|\leq \varepsilon|f(t)|^2+C_\varepsilon|t|^{\frac{p}{2}}$ for all $t \in \mathbb{R}$.
 \item[$(16)$] Assume  $(g_1)$,  $(g_3)$ and $(g_6)$ hold. Let $r\in(q,+\infty)$ if $N=2$, and $r\in(q,2^\ast)$ if $N\geq 3$, then for any $L>0$ there exists $C_L>0$ such that
$G[f(t)]\geq L |t|^{q}-C_L|t|^{r}~~\hbox{for}~~\hbox{all}~~t\in \mathbb{R}.$
\end{itemize}
\end{lemma}
\par
\begin{proof}
$(1)$-$(13)$ can be found in \cite{Wu2014JDE}. Here we only give the proof for $(16)$ because $(14)$-$(15)$ is simple.
For any $L>0$, by $(g_6)$ and $(8)$,  there exists $0<\delta=\delta(L)<1$ such that
 $$G[f(t)]>\frac{L}{C}\bigg|f(t)\bigg|^{q}\geq L|t|^{q}$$
 for any $0<|t|\leq \delta$.
From $(g_3)$ , $(5)$ and $r>q>\frac{p}{2}$, we have
 $$\lim_{t\rightarrow\infty}\frac{G[f(t)]}{-L|t|^{q}+|t|^{r}}\geq\lim_{t\rightarrow\infty}\frac{-|f(t)|^{p}}{-L|t|^{q}+|t|^{r}}
=\lim_{t\rightarrow\infty}\frac{-|t|^{\frac{p}{2}}}{-L|t|^{q}+|t|^{r}}=0.$$
Then there exists $M=M(L)>0$ such that $$G[f(t)]>L|t|^{q}-|t|^{r}$$ for any $|t|> M$.
For $\delta\leq|t|\leq M$, by $(g_1)$,  there exists $C(L)>0$  such that $$G[f(t)]\geq-C(L)\geq-C(L)\delta^{-r}|t|^r.$$
Consequently, for $t\in \mathbb{R}$, $$G[f(t)]\geq L|t|^{q}-\big(1+C(L)\delta^{-r}\big)|t|^{r}:= L|t|^{q}-C_L|t|^{r}.$$
\end{proof}

After making the change of variables, we consider the functional
\begin{equation*}
J^m(\lambda, v)={1\over 2}\|\nabla
v\|_2^2+{ e^\lambda \over2}\big(\|f(v)\|^2_2-m\big)-\int_{\mathbb{R}^N} G[f(v)]
\end{equation*}
and the corresponding semilinear problem
\begin{equation*}
-\Delta v +e^\lambda\big(f(v)f^\prime(v)-m\big)=g[f(v)]f^\prime(v).
\end{equation*}
Let us recall the principle of symmetric criticality (see  \cite{Willem1996book}) as follows.
 \begin{lemma}\label{LemC1}
Assume that the action of the topological group $G$ on the Hilbert space $X$ is isometric. If $\varphi\in C^1(X,\mathbb{R})$ is invariant and if $u$ is a critical point of $\varphi$ restricted to $Fix(G)$, then $u$ is a critical point of $\varphi$.
\end{lemma}
\begin{lemma}\label{LemCC1} Assume that $(g_1)-(g_4)$ hold. Then $J^m \in C^1(\mathbb{R}\times H^1(\mathbb{R}^N),\mathbb{R})$. Moreover, if $(\lambda,v)$ is a critical point for $J^m|_{\mathbb{R}\times H_r^1(\mathbb{R}^N)}$, then $(e^\lambda,f(v))$ solves equation $(P_{\mu,m})$. \end{lemma}
 \begin{proof}
 $J^m \in C^1(\mathbb{R}\times H^1(\mathbb{R}^N),\mathbb{R})$ directly follows from $(g_1)-(g_4)$ and Lemma  \ref{LemA1}. Suppose that $(\lambda,v)$ is a critical point for $J^m|_{\mathbb{R}\times H_r^1(\mathbb{R}^N)}$, by Lemma \ref{LemC1}, $(\lambda,v)$ is also a critical point for  $J^m$. Consequently,
\begin{align*}
0=\langle \partial_v J^m(\lambda,v), \varphi\rangle&=\int_{\mathbb{R}^N}\nabla v \nabla \varphi+ e^\lambda\int_{\mathbb{R}^N}f(v)f^\prime(v)\varphi
-\int_{\mathbb{R}^N}g[f(v)]f^\prime(v)\varphi,~~~~~~~~\forall \varphi \in C_0^\infty(\mathbb{R}^N)
\end{align*}
and
\begin{align*}
 0=\partial_\lambda J^m(\lambda,v)={ e^\lambda \over2}\big(\|f(v)\|_2^2-m\big).
\end{align*}
Let  $u=f(v)$ and $\varphi=\frac{\psi}{f^\prime(v)}$. Then
\begin{align*}
\int_{\mathbb{R}^N}(1+2u^2)\nabla u \nabla\psi +\int_{\mathbb{R}^N}2|\nabla u|^2u\psi + e^\lambda\int_{\mathbb{R}^N}u\psi
-\int_{\mathbb{R}^N}g(u)\psi =0,~~~~~~~~\forall \psi \in C_0^\infty(\mathbb{R}^N)
\end{align*}
and
\begin{align*}
{ e^\lambda \over2}\bigg(\|u\|_2^2-m\bigg)=0.
\end{align*}
Consequently,  $(\mu,u)$ solves equation $(P_{\mu,m})$ with $\mu=e^\lambda$.
\end{proof}

Now, we introduce the Pohozaev functional as follows:
\begin{equation*}
P(\lambda,v)={N-2\over 2}\|\nabla
v\|_2^2+{ Ne^\lambda  \over2}\|f(v)\|_2^2-N\int_{\mathbb{R}^N}
G[f(v)].
\end{equation*}
For $c\in \mathbb{R}$, let
$$K^m_c=\{(\lambda,v)\in\mathbb{R}\times H_r^1(\mathbb{R}^N):J^m(\lambda,v)=c,~\partial_\lambda J^m(\lambda,v)=0,~\partial_v J^m(\lambda,v)=0,~P(\lambda,v)=0\}.$$

\begin{Def}
For $c\in \mathbb{R}$, we call $J^m(\lambda,v)$ satisfies $(PSP)_c$ condition, if  any sequence  $\{(\lambda_n,v_n)\}\subset \mathbb{R}\times H_r^1(\mathbb{R}^N)$ with
\begin{equation}\label{eqB8}J^m(\lambda_n,v_n)\rightarrow c,\end{equation}
\begin{equation}\label{eqB9}\partial_\lambda J^m(\lambda_n,v_n)\rightarrow 0,\end{equation}
\begin{equation}\label{eqB10}\partial_v J^m(\lambda_n,v_n)\rightarrow 0~~\hbox{in}~~H_r^1(\mathbb{R}^N)^\ast,\end{equation}
\begin{equation}\label{eqB11}P(\lambda_n,v_n)\rightarrow0\end{equation}

as $n\rightarrow\infty$, then $\{(\lambda_n,v_n)\}$ has a strongly convergent subsequence in $\mathbb{R}\times H_r^1(\mathbb{R}^N)$.
\end{Def}
\begin{lemma}\label{LemA7} Assume that $(g_1)-(g_4)$ hold. Then $J^m(\lambda,v)$ satisfies $(PSP)_c$ condition for all $c<0$.
\end{lemma}
\begin{proof}
If $\{(\lambda_n,v_n)\}\subset \mathbb{R}\times H_r^1(\mathbb{R}^N)$ satisfying \eqref{eqB8}--\eqref{eqB11}, then
$$\frac{N}{2}me^{\lambda_n}=P(\lambda_n,v_n)-NJ^m(\lambda_n,v_n)+\|\nabla v_n\|_2^2\geq P(\lambda_n,v_n)-NJ^m(\lambda_n,v_n),$$
implies that $\{\lambda_n\}$ is bounded from below. Since  $$\partial_\lambda J^m(\lambda_n,v_n)={ e^{\lambda_n} \over2}\big(\|f(v_n)\|_2^2-m\big)\rightarrow0,$$
one has   $\|f(v_n)\|_2^2\rightarrow m$.
Next, we prove that $\{\nabla v_n\}$ is bounded in $L^2(\mathbb{R}^N)$. We argue by contradiction that
\begin{equation}\label{eqB13}\|\nabla v_n\|_2^2\rightarrow \infty.\end{equation}
For small $\varepsilon>0,$ by Lemma \ref{LemA1} and the Gagliardo--Nirenberg inequality, there exists $C_\varepsilon>0$ such that
\begin{equation}\aligned\label{eqB15}\int_{\mathbb{R}^N}\big|g[f(v_n)]f^{\prime}(v_n)v_n\big|\leq& C_\varepsilon\|f(v_n)\|_2^2+\varepsilon\|f(v_n)\|_{p}^{p}\\
&\leq C_\varepsilon\|f(v_n)\|_2^2+\varepsilon\|f(v_n)\|_2^{\theta p}\|f(v_n)\|_{22^\ast}^{(1-\theta)p}\\
&\leq C_\varepsilon\|f(v_n)\|_2^2+\varepsilon C\|f(v_n)\|_2^{\theta p}\|v_n\|_{2^\ast}^{\frac{(1-\theta)p}{2}}\\
&\leq C_\varepsilon\|f(v_n)\|_2^2+\varepsilon C_N\|f(v_n)\|_2^{\theta p}\|\nabla v_n\|_2^{\frac{(1-\theta)p}{2}}\\
&= C_\varepsilon\|f(v_n)\|_2^2+\varepsilon C_N\|f(v_n)\|_2^{\frac{4}{N}}\|\nabla v_n\|_2^2,
\endaligned\end{equation}
where $\theta=\frac{1}{N+1}$.
Similarly, one has
\begin{equation}\label{eqB160}
 \|v_n\|_2^2\leq C\|f(v_n)\|_2^2+C\|f(v_n)\|_{p}^{p}\leq C\|f(v_n)\|_2^2+ C_N\|f(v_n)\|_2^{\frac{4}{N}}\|\nabla v_n\|_2^2.
\end{equation}
 For large $n$, from \eqref{eqB15}-\eqref{eqB160} and Lemma \ref{LemA1}, we have
 \begin{equation}\aligned \label{eqB16}
 &2C_N\|\nabla v_n\|_2\\
 &\geq\sqrt{\|\nabla v_n\|_2^2+C\|f(v_n)\|_2^2+ C_N\|f(v_n)\|_2^{\frac{4}{N}}\|\nabla v_n\|_2^2}\\
 &\geq\|\partial_v J^m(\lambda_n,v_n)\|\sqrt{\|\nabla v_n\|_2^2+\|v_n\|_2^2}\\
 &\geq \langle\partial_v J^m(\lambda_n,v_n),v_n\rangle\\
 &= \|\nabla v_n\|_2^2+ e^{\lambda_n}\int_{\mathbb{R}^N}f(v_n)f^{\prime}(v_n)v_n-\int_{\mathbb{R}^N}g[f(v_n)]f^{\prime}(v_n)v_n\\
 &\geq\|\nabla v_n\|_2^2+\frac{1}{2}e^{\lambda_n}\|f(v_n)\|_2^2-C_\varepsilon\|f(v_n)\|_2^2-\varepsilon C_N\|f(v_n)\|_2^{\frac{4}{N}}\|\nabla v_n\|_2^2\\
 &\geq\big(1-\varepsilon C_N(2m)^{\frac{2}{N}}\big)\|\nabla v_n\|_2^2-2C_\varepsilon m,
\endaligned\end{equation}
which contradicts \eqref{eqB13} if we choose sufficiently small $\varepsilon$ . Moreover, the proof of \eqref{eqB16} implies that $\{\lambda_n\}$ is bounded from above.
By claim $2$ of Lemma 2.2 in \cite{YangTang2019JMP}, up to a subsequence, we may assume that $\lambda_n\rightarrow\lambda_0$, and
\begin{equation*}
\label {eqD1}
\begin{cases}
v_n\rightharpoonup v_0~~\hbox{in}~~H_r^1(\mathbb{R}^N),\\
v_n\rightarrow v_0~~\hbox{in}~~L^r(\mathbb{R}^N)~~\hbox{for}~~2<r<2^\ast,\\
v_n\rightarrow v_0~~\hbox{in}~~L_{loc}^s(\mathbb{R}^N)~~\hbox{for}~~2\leq s<2^\ast,\\
v_n\rightarrow v_0~~a.e.~~\hbox{in}~~\mathbb{R}^N.
\end{cases}
\end{equation*}
For small $\varepsilon>0$, by Lemma \ref{LemA1}, there exists $C_\varepsilon>0$ such that
\begin{equation}\aligned \label{eqB166}
&\bigg|\int_{\mathbb{R}^N}\big(g[f(v_n)]f^{\prime}(v_n)-g[f(v_0)]f^{\prime}(v_0)\big)\big(v_n-v_0\big)\bigg|\\
&\leq \int_{\mathbb{R}^N}\big(\varepsilon |f(v_n)||f^{\prime}(v_n)|+C_\varepsilon |f(v_n)|^{p-1}|f^{\prime}(v_n)|+\varepsilon |f(v_0)||f^{\prime}(v_0)|+C_\varepsilon |f(v_0)|^{p-1}|f^{\prime}(v_0)|\big)\big|v_n-v_0\big|\\
&\leq \int_{\mathbb{R}^N}\big(\varepsilon |v_n|+C_\varepsilon |v_n|^{\frac{p-2}{2}}+\varepsilon |v_0|+C_\varepsilon |v_0|^{\frac{p-2}{2}}\big)\big|v_n-v_0\big|\\
&\leq \varepsilon C+C_\varepsilon\bigg(\big\|v_n\big\|_{\frac{p}{2}}^{\frac{p-2}{2}}+\big\|v_0\big\|_{\frac{p}{2}}^{\frac{p-2}{2}}\bigg)\big\|v_n-v_0\big\|_{\frac{p}{2}}\\
&= \varepsilon C+o_n(1).
\endaligned\end{equation}
 Now, by view of  Lemma 2.6 in \cite{YangWangZhao2014NA}, we have
\begin{equation}\aligned\label{eqB19}o_n(1)&=\big\langle\partial_v J^m(\lambda_n,v_n)- \partial_vJ^m(\lambda_n,v_0),v_n-v_0\big\rangle\\
&= \|\nabla v_n-\nabla v_0\|_2^2+ e^{\lambda_n}\int_{\mathbb{R}^N}\big(f(v_n)f^{\prime}(v_n)-f(v_0)f^{\prime}(v_0)\big)\big(v_n-v_0\big)\\
&~~~-\int_{\mathbb{R}^N}\big(g[f(v_n)]f^{\prime}(v_n)-g[f(v_0)]f^{\prime}(v_0)\big)\big(v_n-v_0\big)\\
&\geq C\|v_n-v_0\|^2+o_n(1),
\endaligned\end{equation}
which implies that $v_n\rightarrow v$ in $H_r^1(\mathbb{R}^N)$, and the proof is  completed.
\end{proof}
\subsection{ Deformation Lemma}
Following \cite{Tanaka2019ANS}, we introduce the  augmented functional to construct a deformation flow as follows:
\begin{equation*}
\F^m(\theta,\lambda,v)={e^{(N-2)\theta}\over 2}\|\nabla
v\|^2+{e^\lambda\over 2}\bigg(e^{N\theta}\|f(v)\|^2_2-m\bigg)
-e^{N\theta}\int_{\mathbb{R}^N}
G[f(v)].
\end{equation*}
Without  causing symbol confusion, in this subsection we denote $J^m$ and $\F^m$ by $J$ and $\F$, respectively. By a direct  calculation, we obtain the following Lemma.
\begin{lemma}\label{LemA8}
For any $(\theta,\lambda,v)\in \mathbb{R}^2\times H_r^1(\mathbb{R}^N)$, $\varphi \in  H_r^1(\mathbb{R}^N)$ and $\beta>0$,
$$\F(\theta,\lambda,v)=J\big(\lambda,v(e^{-\theta}x)\big),$$
$$~~~~\partial_\theta \F(\theta,\lambda,v)=P\big(\lambda,v(e^{-\theta}x)\big),$$
$$~~~~\partial_\lambda \F(\theta,\lambda,v)=\partial_\lambda J\big(\lambda,v(e^{-\theta}x)\big),$$
$$~~~~\big\langle\partial_v \F(\theta,\lambda,v),\varphi(x)\big\rangle=\big\langle\partial_v J\big(\lambda,v(e^{-\theta}x)\big),\varphi\big(e^{-\theta}x\big)\big\rangle,$$
$$~~~~~~~~~~\F(\theta+\beta,\lambda,v(e^\beta x))=\F(\theta,\lambda,v(x)).$$
\end{lemma}
We introduce a metric on $M:=\mathbb{R}^2\times H_r^1(\mathbb{R}^N)$ by
$$\|(\alpha,\nu,h)\|^2_{(\theta,\lambda,v)}=\big|(\alpha,\nu,\|h(e^{-\theta}\cdot)\|)\big|^2$$
for any $(\alpha,\nu,h)\in T_{(\theta,\lambda,v)}M$. Then $M$ be a Hilbert manifold. We also denote the dual norm on $T^\ast_{(\theta,\lambda,v)}M$ by $\|\cdot\|_{(\theta,\lambda,v),\ast}$.

Let  $$\D \F=(\D _\theta \F,\D _\lambda \F,\D _v\F).$$
By a direct computation, we have
$$\|\D \F(\theta,\lambda,v)\|_{(\theta,\lambda,v),\ast}=\big|P(\lambda,v(e^{-\theta}\cdot))\big|^2+\big|\partial_\lambda J\big(\lambda,v(e^{-\theta}\cdot)\big)\big|^2+\big\|\partial_v J\big(\lambda,v(e^{-\theta}\cdot)\big)\big\|^2.$$

We introduce a natural distance  in Sobolev space $\mathbb{R}^2\times H_r^1(\mathbb{R}^N)$ as follows:
\begin{align*}
&\hbox{dist}\big((\theta_1,\lambda_1,v_1),(\theta_2,\lambda_2,v_2)\big)\\
&=\inf\bigg\{\int_0^1||\dot{\sigma}(t)||dt:\sigma(t)\in C^1([0,1],\mathbb{R}^2\times H_r^1(\mathbb{R}^N)),\sigma(0)=(\theta_1,\lambda_1,v_1),\sigma(1)=(\theta_2,\lambda_2,v_2)\bigg\}.
\end{align*}

For $c\in \mathbb{R}$, we denote
$$\widetilde{K^m_c}=\{(\theta,\lambda,v)\in\mathbb{R}^2\times H_r^1(\mathbb{R}^N):\F(\theta,\lambda,v)=c,~\D \F(\theta,\lambda,v)=(0,0,0)\}.$$
By Lemma \ref{LemA8}, we have
$$\widetilde{K^m_c}=\{(\theta,\lambda,v(e^\theta x)):\theta\in\mathbb{R},~(\lambda,v)\in K^m_c\}.$$
For $\rho>0$, $\widetilde{A}\in\mathbb{R}^2\times H_r^1(\mathbb{R}^N)$, $c\in\mathbb{R}$, let us denote
$$\widetilde{N}_\rho(\widetilde{A})=\big\{(\theta_1,\lambda_1,v_1)\in\mathbb{R}^2\times H_r^1(\mathbb{R}^N):\inf_{(\theta_2,\lambda_2,v_2)\in\widetilde{A}}\hbox{dist}\big((\theta_1,\lambda_1,v_1),(\theta_2,\lambda_2,v_2)\big)<\rho\big\},$$
$$\F^c=\{(\theta,\lambda,v)\in\mathbb{R}^2\times H_r^1(\mathbb{R}^N):\F(\theta,\lambda,v)\leq c\}.$$

The following Lemma can be founded in (\cite{Tanaka2019ANS}.

\begin{lemma}\label{LemA9} If $c<0$, then for any $\overline{\varepsilon}>0$ and $\rho>0$ there exist $\varepsilon\in(0,\overline{\varepsilon})$ and a continuous map $\widetilde{\eta}(t,\theta,\lambda,v):[0,1]\rightarrow \mathbb{R}^2\times H_r^1(\mathbb{R}^N)$ such that
\begin{itemize}
\item[$(1)$] $\widetilde{\eta}(0,\theta,\lambda,v)=(\theta,\lambda,v)$ for all $(\theta,\lambda,v)\in \mathbb{R}^2\times H_r^1(\mathbb{R}^N).$
\item[$(2)$]  $\widetilde{\eta}(0,\theta,\lambda,v)=(\theta,\lambda,v)$ for all $(\lambda,v)\in \F^{c-\overline{\varepsilon}}$ and $t\in[0,1].$
\item[$(3)$]  $\F(\widetilde{\eta}(t,\theta,\lambda,v))\leq \F(\theta,\lambda,v)$ for all $(t,\theta,\lambda,v)\in[0,1]\times\mathbb{R}^2\times H_r^1(\mathbb{R}^N).$
\item[$(4)$]  $\widetilde{\eta}(1,\F^{c+{\varepsilon}}\setminus\widetilde{N}_\rho(\widetilde{K}_c))\subset \F^{c-{\varepsilon}},$  $\widetilde{\eta}(1,\F^{c+{\varepsilon}})\subset \F^{c-{\varepsilon}}\cup \widetilde{N}_\rho(\widetilde{K}_c).$
\item[$(5)$]  If $\widetilde{K}_c=\emptyset$, then $\widetilde{\eta}(1,\F^{c+{\varepsilon}})\subset \F^{c-{\varepsilon}}$.
\item[$(6)$] Let $\widetilde{\eta}(t,\theta,\lambda,v)=(\widetilde{\eta_1}(t,\theta,\lambda,v),\widetilde{\eta_2}(t,\theta,\lambda,v),\widetilde{\eta_3}(t,\theta,\lambda,v))$, we have
    $$\widetilde{\eta_1}(t,\theta,\lambda,-v)=\widetilde{\eta_1}(t,\theta,\lambda,v),$$
    $$\widetilde{\eta_2}(t,\theta,\lambda,-v)=\widetilde{\eta_2}(t,\theta,\lambda,v),$$
    $$\widetilde{\eta_3}(t,\theta,\lambda,-v)=-\widetilde{\eta_3}(t,\theta,\lambda,v)$$
    for all $(t,\theta,\lambda,v)\in[0,1]\times\mathbb{R}^2\times H_r^1(\mathbb{R}^N).$
\end{itemize}
\end{lemma}
\begin{lemma}( \textbf {Deformation Lemma})\label{LemA11} If $c<0$, then
\begin{itemize}
\item[(\expandafter{\romannumeral1})] $K_c^m$ is compact in $\mathbb{R}\times H_r^1(\mathbb{R}^N)$ and $K_c^m\cap\{\mathbb{R}\times\{0\}\}=\emptyset$.
\item[$(\expandafter{\romannumeral2})$]  For any open neighborhood $N$ of $K_c^m$ and $\overline{\varepsilon}>0$, there exist $\varepsilon\in(0,\overline{\varepsilon})$ and a continuous map $\eta(t,\lambda,v):[0,1]\rightarrow \mathbb{R}\times H_r^1(\mathbb{R}^N)$ such that
\item[$(1)$]  $\eta(0,\lambda,v)=(\lambda,v)$ for all $(\lambda,v)\in \mathbb{R}\times H_r^1(\mathbb{R}^N)$.
\item[$(2)$]  $\eta(t,\lambda,v)=(\lambda,v)$ for all $(\lambda,v)\in J^{c-\overline{\varepsilon}}$ and $t\in[0,1].$
\item[$(3)$]  $J(\eta(t,\lambda,v))\leq J(\lambda,v)$ for all $(t,\lambda,v)\in[0,1]\times\mathbb{R}\times H_r^1(\mathbb{R}^N).$
\item[$(4)$]  $\eta(1,J^{c+{\varepsilon}}\setminus N)\subset J^{c-{\varepsilon}},$   $\eta(1,J^{c+{\varepsilon}})\subset J^{c-{\varepsilon}}\cup N.$
\item[$(5)$]  If $K_c^m=\emptyset$, then $\eta(1,J^{c+{\varepsilon}})\subset J^{c-{\varepsilon}}$.
\item[$(6)$] Let $\eta(t,\lambda,v)=(\eta_1(t,\lambda,v),\eta_2(t,\lambda,v))$, we have
    $$\eta_1(t,\lambda,-v)=\eta_1(t,\lambda,v),$$
    $$\eta_2(t,\lambda,-v)=-\eta_2(t,\lambda,v),$$
    for all $(t,\lambda,v)\in[0,1]\times\mathbb{R}\times H_r^1(\mathbb{R}^N).$
\end{itemize}
\end{lemma}
\begin{proof}
Consider the  following maps:
$$\pi:\mathbb{R}^2\times H_r^1(\mathbb{R}^N)\rightarrow \mathbb{R}\times H_r^1(\mathbb{R}^N),~~(\theta,\lambda,v(x))\mapsto\big(\lambda,v(e^{-\theta}x)\big), $$
$$l:\mathbb{R}\times H_r^1(\mathbb{R}^N)\rightarrow \mathbb{R}^2\times H_r^1(\mathbb{R}^N),~~(\lambda,v(x))\mapsto(0,\lambda,v(x)).$$
By Lemma \ref{LemA8}, we have
$$\pi(l(\lambda,v))=(\lambda,v)~~~~~~~~~~~~~~~~~~\hbox{for}~~\hbox{all}~~(\lambda,v)\in\mathbb{R}\times H_r^1(\mathbb{R}^N),$$
$$l(\pi(\theta,\lambda,v))=\big(0,\lambda,v(e^{-\theta}x)\big)~~~~~~~~\hbox{for}~~\hbox{all}~~(\theta,\lambda,v)\in\mathbb{R}^2\times H_r^1(\mathbb{R}^N),$$
$$\F(\theta,\lambda,v)=J(\pi(\theta,\lambda,v))~~~~~~~~~~~~~\hbox{for}~~\hbox{all}~~(\theta,\lambda,v)\in\mathbb{R}^2\times H_r^1(\mathbb{R}^N).$$
  $\pi(\widetilde{K_c})=K_c.$
From those results above and Lemma \ref{LemA9}, deformation Lemma follows. We can refer the  readers to Lemma 3.5 in \cite{Myself2021JMAA} for details.
\end{proof}
\section{Minimax methods}
\subsection{Construction of multidimensional odd path}
In what follows, we set
\begin{equation*}
J_\lambda(v)={1\over 2}\|\nabla
v\|_2^2+{ e^\lambda \over2}\|f(v)\|_2^2-\int_{\mathbb{R}^N}
G[f(v)],
\end{equation*}
and
\begin{equation*}
\lambda_0=
\begin{cases}
\ln\bigg(2\sup\limits_{s\neq 0}\frac{G(s)}{s^2}\bigg)~~~~\hbox{if}~~\sup\limits_{s\neq 0}\frac{G(s)}{s^2}<\infty,\\
\infty~~~~~~~~~~~~~~~~~~~~~~\hbox{if}~~\sup\limits_{s\neq 0}\frac{G(s)}{s^2}=\infty.
\end{cases}
\end{equation*}
\begin{Def}
For $k=1$ and $\lambda \in (-\infty,\lambda_0)$, let
$$a_1(\lambda)=\inf_{\gamma\in{\Gamma}_1(\lambda)} \max_{\xi \in [0,1]}{J}_\lambda(\gamma(\xi)), $$ where
$$\Gamma_1(\lambda)=\{\gamma\in C([0,1],H_r^1(\mathbb{R}^N)):\gamma(0)=0, {J}_\lambda(\gamma(1))<0\};$$
For  $k\geq2$ and $\lambda \in (-\infty,\lambda_0)$, let
$$a_k(\lambda)=\inf_{\gamma\in{\Gamma}_k(\lambda)} \max_{\xi \in D_k}{J}_\lambda(\gamma(\xi)), $$ where
$$\Gamma_k(\lambda)=\{\gamma\in C(D_k,H_r^1(\mathbb{R}^N)):\gamma(-\xi)=-\gamma(\xi)~\hbox{for}~\xi \in D_k, {J}_\lambda(\gamma(\xi))<0~\hbox{for}~\xi \in \partial  D_k
\}.$$
\end{Def}
\begin{lemma}\label{LemAA0}
 Assume that $\lambda \in (-\infty,\lambda_0)$ and $(g_1)-(g_4)$ hold, $\Gamma_1(\lambda)\neq \emptyset.$  In addition $(g_5)$ holds, then  $\Gamma_k(\lambda)\neq \emptyset$ for any $k\geq1.$
\end{lemma}
\begin{proof}
Now, we prove the case for $k\geq2$. Let $H(s):=-\frac{e^\lambda}{2}f^2(s)+G[f(s)]$, for any $\lambda \in (-\infty,\lambda_0)$, by the definition of $\lambda_0,$  there exists $s_\lambda>0$ such that $H(s_\lambda)>0.$  If $(g_1)-(g_5)$ hold, there exists  $\delta>0$ such that
 $$H(s)>0,~\forall s\in (s_\lambda-\delta,s_\lambda+\delta)\cup (-s_\lambda-\delta,-s_\lambda+\delta).$$
 For  $R\gg1$ and $h\in [0,1]$,  let  $$A(R,h)=\{x\in\mathbb{R}^N:R-h\leq|x|\leq R+h\}$$ and
 \begin{equation*}
\chi(R,h;x)=
\begin{cases}
1,~~~~~~~~~~~~~~~~~~~~~~~ \hbox{if}~~x\in A(R,h),\\
0, ~~~~~~~~~~~~~~~~~~~~~~~~\hbox{otherwise}.
\end{cases}
\end{equation*}
 Consider the polyhedron
$$\Sigma:=\{t=(t_1,...,t_k)\in\mathbb{R}^k:\max\limits_{1\leq i\leq k}|t_i|=1\}$$
and define the odd  map  $\gamma:\Sigma\rightarrow H_r^1(\mathbb{R}^N)$ by
 $$\gamma(t)(\cdot)=\sum_{i=1}^{k} \hbox{sgn}(t_i)\big(s_\lambda+(|t_i|-1)\delta\big)\chi(2ki,|t_i|;\cdot).$$
For all $t=(t_1,...,t_k)\in \sum$, we have $\chi(2ki,|t_i|;x)\cap\chi(2kj,|t_j|;\cdot)=0$ if $i\neq j$ and
\begin{equation}\aligned\label{eqBb1}
\int_{\mathbb{R}^N}H(\gamma(t)(x))&=\sum_{i=1}^{k}H\bigg(\hbox{sgn}(t_i)\big(s_\lambda+(|t_i|-1)\delta\big)\bigg)\int_{\mathbb{R}^N}\chi(2ki,|t_i|;x)\\
&\geq H(s_\lambda)\int_{\mathbb{R}^N} \chi(2k{i_0},1;\cdot)\geq C_0.
\endaligned\end{equation}
Let
\begin{equation*}
\chi_{\varepsilon}(R,h;x)=
\begin{cases}
1,~~~~~~~~~~~~~~~~~~~~~~~~~~ \hbox{if}~~R-h\leq|x|\leq R+h,\\
1-\frac{1}{\varepsilon h}\big(R-h-|x|\big),~~\hbox{if}~R-h-\varepsilon h\leq|x|\leq R-h,\\
1-\frac{1}{\varepsilon h}\big(|x|-R-h\big),~~\hbox{if}~~R+h\leq|x|\leq R+h+\varepsilon h,\\
0, ~~~~~~~~~~~~~~~~~~~~~~~~~~~\hbox{otherwise}.
\end{cases}
\end{equation*}
It is easy to see that
\begin{itemize}
\item[(\expandafter{\romannumeral1})] $\chi_{\varepsilon}(R,h;\cdot)\in H_r^1(\mathbb{R}^N);$
\item[(\expandafter{\romannumeral2})]For fixed $R>0$ and small $\varepsilon>0$, $\chi_{\varepsilon}(R,h;\cdot)~~\hbox{is}~~
\hbox{continuous}~~\hbox{in}~~H_r^1(\mathbb{R}^N)~~\hbox{for}~~h\in[0,1];$
\item[(\expandafter{\romannumeral3})]$\chi_{\varepsilon}(R,h;\cdot)\rightarrow \chi(R,h;\cdot)~~\hbox{in}~~L^r(\mathbb{R}^N)~~\hbox{as}~~\varepsilon\rightarrow0~~\hbox{uniformly}~~\hbox{for}~~h\in[0,1]~~\hbox{and}~~R\leq C;$
\item[(\expandafter{\romannumeral4})] $\chi_{\varepsilon}(R,h;\cdot)\rightarrow 0~~\hbox{in}~~L^r(\mathbb{R}^N)~~\hbox{as}~~h\rightarrow0~~\hbox{uniformly}~~\hbox{for}~~\hbox{small}~~\varepsilon>0~~\hbox{and}~~R\leq C.$
\item[(\expandafter{\romannumeral5})]For each $t=(t_1,...,t_k)\in\sum$,  $\hbox{supp}(\chi_{\varepsilon}(2ki,|t_i|;\cdot))\cap\hbox{supp}(\chi_{\varepsilon}(2kj,|t_j|;\cdot))=\emptyset~~\hbox{if}~~i\neq j$.
\end{itemize}
Let $\gamma_\varepsilon:\sum\rightarrow H_r^1(\mathbb{R}^N)$ defined by
$$\gamma_\varepsilon(t)(\cdot)=\sum_{i=1}^{k} \hbox{sgn}(t_i)\big(s_\lambda+(|t_i|-1)\delta\big)\chi_{\varepsilon}(2ki,|t_i|;\cdot).$$
From (\expandafter{\romannumeral1})--(\expandafter{\romannumeral2}), (\expandafter{\romannumeral4}), $\gamma_\varepsilon$ is odd and continuous. We observe that $\int_{\mathbb{R}^N}H(v)$ is continuous in  $L^r(\mathbb{R}^N)$ for $r\in [2,2^\ast]$. By using  \eqref{eqBb1}, (\expandafter{\romannumeral3}) and (\expandafter{\romannumeral5}), there holds
\begin{equation}\aligned\label{eqBb2}
\int_{\mathbb{R}^N}H(\gamma_\varepsilon(t)(x))&=\sum_{i=1}^{k}H\bigg(\hbox{sgn}(t_i)\big(s_\lambda+(|t_i|-1)\delta\big)\bigg)\int_{\mathbb{R}^N}\chi_{\varepsilon}(2ki,|t_i|;\cdot)\\
&\geq H(s_\lambda)\int_{\mathbb{R}^N} \chi_\varepsilon(2k{i_0},1;\cdot)\\
& \rightarrow H(s_\lambda)\int_{\mathbb{R}^N} \chi(2k{i_0},1;\cdot)\\
&\geq C_0.
\endaligned\end{equation}
For small $\varepsilon$, since $\Sigma$ is homeomorphic to $\partial D_k$, we may assume that
$$\int_{\mathbb{R}^N}H(\gamma_\varepsilon(t))>C~~\hbox{for}~~\hbox{all}~~t\in\partial D_k.$$
Note that
\begin{align*}
J_\lambda\big(\gamma_\varepsilon(t)(L^{-1}\cdot)\big)&={1\over 2}L^{N-2}\|\nabla
\gamma_\varepsilon(t)\|_2^2-L^N\int_{\mathbb{R}^N}H(\gamma_\varepsilon(t)(\cdot)),
\end{align*}
there exists  some large $L>0$ such that
$$J_\lambda\big(\gamma_\varepsilon(t)(L^{-1}\cdot)\big)<0~~\hbox{for}~~\hbox{all}~~t\in\partial D_k.$$
 Let
$$\eta(s,t)(\cdot)=s\gamma_\varepsilon(t)(L^{-1}\cdot).$$  Then  $D_k=\{st|s\in[0,1],t\in\sum\}$, and hence $\eta\in \Gamma_k(\lambda).$

If $(g_1)-(g_4)$ hold, there exist $s_\lambda>0$ and  $\delta>0$ such that
 $$H(s)>0,~\forall s\in (s_\lambda-\delta,s_\lambda+\delta).$$  For $t\in[0,1]$, we set $\gamma_\varepsilon(t)(\cdot)=\big(s_\lambda+(t-1)\delta\big)\chi_{\varepsilon}(R,t;\cdot).$
 Then $$\int_{\mathbb{R}^N}H(\gamma_\varepsilon(1))=H(s_\lambda)\int_{\mathbb{R}^N}\chi_{\varepsilon}(R,1;\cdot)\geq C.$$
 Take large $L>0$ and let $\eta(t)(\cdot)=\gamma_\varepsilon(t)(L^{-1}\cdot)$, we have  $\eta\in \Gamma_1(\lambda).$
\end{proof}

For any $\varepsilon>0$, by Lemma \ref{LemA1} and the Sobolev  theorem,  there exists $C_\varepsilon>0$ such that
 $$\bigg|\int_{\mathbb{R}^N}
G[f(v)]\bigg|\leq\varepsilon\big(\|\nabla
v\|_2^2+\|f(v)\|_2^2\big)+C_\varepsilon
\big(\|\nabla
v\|_2^2+\|f(v)\|_2^2\big)^{\frac{2^*}{2}}.$$
For all $v\in  \partial \Sigma_\rho$ and small $\varepsilon>0$, we have
\begin{equation}\aligned\label{eqBbb1}
J_\lambda(v)&={1\over 2}\|\nabla
v\|_2^2+{ e^\lambda \over2}\|f(v)\|_2^2-\int_{\mathbb{R}^N}
G[f(v)]\\
&\geq\big(C_\lambda-\varepsilon\big)\big(\|\nabla
v\|_2^2+\|f(v)\|_2^2\big)-C_\varepsilon
\big(\|\nabla
v\|_2^2+\|f(v)\|_2^2\big)^{\frac{2^*}{2}}\\
&\geq {C_\lambda\over 2}\rho-C_\lambda^\prime\rho^{\frac{2^*}{2}}.
\endaligned\end{equation}
Then for any small $\rho>0$, there exists $\alpha_\rho>0$ such that $J_\lambda(v)\geq\alpha_\rho$ for $v\in  \partial \Sigma_\rho$.
By the definition of $a_k(\lambda)$, we get that  $0\leq  m_k\leq m_{k+1},~\forall~k\in\mathbb{N}^\ast,$ where
  $$m_k:=2\inf\limits_{\lambda\in(-\infty,\lambda_0)}\frac{a_k(\lambda)}{e^\lambda}.$$

Furthermore, we have the following result:
\par
\begin{lemma}\label{LemA2}
 Assume that $\lambda \in (-\infty,\lambda_0)$ and $(g_1)-(g_4)$ hold. For any $k\in\mathbb{N}$,
\begin{itemize}
\item[$(1)$] $m_k=0$, if $(g_6)$ holds.
\item[$(2)$] $m_k>0$, if  $\limsup\limits_{s\rightarrow0}\frac{|g(s)|}{|s|^{q-2}s}<\infty.$
\end{itemize}
\end{lemma}
\begin{proof}
 Denote  $E_k=\hbox{span}\{e_1,e_2,...,e_k\}$, where $\{e_i\}$ is a sequence  of orthonormal basis in $H_r^1(\mathbb{R}^N)$.  Since $E_k$ is  a finite dimensional space, for any  $2\leq s\leq 2^\ast$ there exists $c_s>0$ such that
\begin{equation} \label{eqB0} c_s\|v\|_s \leq \|v\|\leq  c_s^{-1}\|v\|_s,~~~~\forall v\in E_k.\end{equation}
Thus, we may choose some $R>0$ such that
$$\frac{1}{2}\|v\|^2-\frac{1}{s}\|v\|^s_s<0,~~~~\forall v\in E_k\setminus B_R(0).$$
Let $\eta:\mathbb{R}^k\rightarrow E_k$ defined by
$$\eta(\xi)=R\sum_{i=1}^k\xi_i e_i,$$
where $\xi=(\xi_1,...,\xi_k) \in \mathbb{R}^k$.  For any $ \xi\in \partial D_k$, we have
\begin{equation} \label{eqB1}\frac{1}{2}\|\eta(\xi)\|^2-\frac{1}{q}\|\eta(\xi)\|^{q}_{q}<0.\end{equation}
Let $\gamma(\xi)(\cdot)=e^{\frac{\lambda}{q-2}}\eta(\xi)(e^{\frac{\lambda}{2}}\cdot)$. For any $ \xi\in \partial D_k$ with  $L>\frac{1}{q}$, it follow  \eqref{eqB0} and Lemma \ref{LemA1}-$(16)$ that
\begin{align*}
J_\lambda(\gamma(\xi))&\leq{1\over 2}\|\nabla
\gamma(\xi)\|_2^2+{e^\lambda \over2}\|\gamma(\xi)\|_2^2-L\|\gamma(\xi)\|_{q}^{q}
+C_L\|\gamma(\xi)\|_{r}^{r}\\
&= e^\lambda\big({1\over 2}\|\eta(\xi)\|^2-L\|\eta(\xi)\|_{q}^{q}
+C_Le^{\frac{r-q}{q-2}\lambda}\|\eta(\xi)\|_r^r\big)\\
&\leq e^\lambda\bigg({1\over 2}\|\eta(\xi)\|^2-\frac{1}{q}\|\eta(\xi)\|^{q}_q+\big(\frac{1}{q}-L\big)\|\eta(\xi)\|^{q}_{q}
+C_Le^{\frac{r-q}{q-2}\lambda}\|\gamma_k(\xi)\|^{r}_{r}\bigg)\\
&\leq e^\lambda\bigg({1\over 2}\|\eta(\xi)\|^2-\frac{1}{q}\|\eta(\xi)\|^{q}+\big(\frac{1}{q}-L\big)(Rc_q)^{q}
+C_Le^{\frac{r-q}{q-2}\lambda}(Rc_r^{-1})^{r}\bigg),\\
\end{align*}
which implies that $\gamma \in \Gamma_k(\lambda)$ as $\lambda\rightarrow -\infty$.  Consequently,
\begin{align*}
\limsup_{\lambda\rightarrow-\infty}\frac{a_k(\lambda)}{e^\lambda}&\leq  \max_{\xi \in D_k}\bigg({1\over 2}\|\eta(\xi)\|^2-L\|\eta(\xi)\|^{q}_{q}\bigg)\\
&\leq  \max_{\xi \in D_k}\bigg({1\over 2}\|\eta(\xi)\|^2-L(c_q)^{q}\|\eta(\xi)\|^{q}\bigg)\\
&\leq  \max_{s\geq0}\bigg({1\over 2}|s|^2-L(c_q)^{q}|s|^{q}\bigg)\\
&=\bigg({1\over 2}|s|^2-L(c_q)^{q}|s|^{q+1}\bigg)\bigg|_{s=\big(Lq(c_q)^q\big)^{-\frac{1}{q-2}}}\\
&\rightarrow0
\end{align*}
 as $L\rightarrow+\infty$, and hence  $(1)$ holds.
If $\limsup\limits_{s\rightarrow0}\frac{|g(s)|}{|s|^{q-2}s}<\infty,$ it follows from $(g_1)$ and $(g_3)$ that
$$G[f(s)]\leq C|s|^q~~\hbox{for}~~\hbox{all}~~s\in \mathbb{R}.$$
For any $\lambda\rightarrow-\infty$ and $\gamma_\lambda \in {\Gamma}_k(\lambda)$, let $\eta_\lambda(\xi)(\cdot)=e^{-\frac{\lambda}{q-2}}\gamma_\lambda(\xi)(e^{-\frac{\lambda}{2}}\cdot)$.
Since $e^{\frac{\lambda}{q-2}}\leq 1$, by Lemma \ref{LemA1} and Sobolev  theorem, one has
\begin{align*}
\max_{\xi \in D_k}J_\lambda(\lambda,\gamma_\lambda(\xi))&=\max_{\xi \in D_k}\bigg({1\over 2}\|\nabla\gamma_\lambda(\xi)\|_2^2+{ e^\lambda \over2}\|f[\gamma_\lambda(\xi)]\|_2^2-\int_{\mathbb{R}^N}
G[f(\gamma_\lambda(\xi))]\bigg)\\
&\geq \max_{\xi \in D_k} \bigg({1\over 2}\|\nabla\gamma_\lambda(\xi)\|_2^2+{ e^\lambda \over2}\|f[\gamma_\lambda(\xi)]\|_2^2-C\|\gamma_\lambda(\xi)\|_q^q\bigg)\\
&\geq e^\lambda\max_{\xi \in D_k} \bigg({1\over 2}\big(\|\nabla\eta_\lambda(\xi)\|_2^2+{1\over 2}\|f[\eta_\lambda(\xi)]\|_2^2-C\|\eta_\lambda(\xi)\|_q^q
\bigg)\\
&\geq e^\lambda\max_{\xi \in D_k} \bigg({1\over 2}\big(\|\nabla\eta_\lambda(\xi)\|_2^2+{1\over 2}\|f[\eta_\lambda(\xi)]\|_2^2-\varepsilon\|f[\eta_\lambda(\xi)]\|_2^2-C_\varepsilon\|f\eta_\lambda(\xi)]\|_{2^*}^{2^*}
\bigg)\\
&\geq e^\lambda\max_{\xi \in D_k}\bigg({1\over 4}\big(\|\nabla\eta_\lambda(\xi)\|_2^2+\|f[\eta_\lambda(\xi)]\|_2^2\big)-C\big(\|\nabla\eta_\lambda(\xi)\|_2^2+\|f[\eta_\lambda(\xi)]\|_2^2\big)^{\frac{2^*}{2}}
\bigg)\\
&= e^\lambda\max_{\rho\geq0}\bigg(\frac{1}{4}\rho-C\rho^{\frac{2^\ast}{2}}\bigg)\\
&=e^\lambda\frac{1}{2N}\bigg(\frac{N-2}{4NC}\bigg)^{\frac{N-2}{2}}
\end{align*}
Since $\gamma_\lambda$ is arbitrary, we have $m_k\geq \frac{1}{2N}\big(\frac{N-2}{4NC}\big)^{\frac{N-2}{2}}$, and the proof is completed.
\end{proof}

\begin{lemma}\label{LemAA2} Assume that $\lambda_0 =\infty$ and $(g_1)-(g_4)$ hold. Then
$\lim\limits_{\lambda\rightarrow+\infty}\frac{a_k(\lambda)}{e^\lambda}=+\infty.$
\end{lemma}
\begin{proof}
 We denote $e^\lambda$ by $\mu$ for  simplicity.  Since $0<a_1(\mu)\leq a_k(\mu)$ for each $k\in \mathbb{N}^\ast$, it is sufficient to show the conclusion for $k=1$.
Let $w_{r,s}(\cdot)=rv(s\cdot)$  for any $v\in H_r^1(\mathbb{R}^N)$,  by Lemma \ref{LemA1}, we have
\begin{equation}\aligned\label{eqBB15}
J_\mu(w_{r,s})&={1\over 2}\|\nabla
w_{r,s}\|_2^2+{\mu \over2}\|f(w_{r,s})\|_2^2-\int_{\mathbb{R}^N}G[f(w_{r,s})]\\
&={1\over 2}r^2s^{2-N}\|\nabla v\|_2^2+{\mu \over2}s^{-N}\|f(rv)\|_2^2-s^{-N}\int_{\mathbb{R}^N}G[f(rv)]\\
&\geq{1\over 2}r^2s^{2-N}\|\nabla v\|_2^2+\big({\mu \over2}-\varepsilon\big)s^{-N}\|f(rv)\|_2^2-C_\varepsilon s^{-N}r^{\frac{p}{2}}\|v\|_{\frac{p}{2}}^{\frac{p}{2}}\\
&\geq{1\over 2}r^2s^{2-N}\|\nabla v\|_2^2+Cr\big({\mu \over2}-C_\varepsilon\big)s^{-N}\|f(v)\|_2^2-\varepsilon s^{-N}r^{\frac{p}{2}}\|v\|_{\frac{p}{2}}^{\frac{p}{2}}\\
&\geq{1\over 2}r^2s^{2-N}\|\nabla v\|_2^2+\bigg[Cr\big({\mu \over2}-C_\varepsilon\big)s^{-N}-\varepsilon s^{-N}r^{\frac{p}{2}}\bigg]\|f(v)\|_2^2- C\varepsilon s^{-N}r^{\frac{p}{2}}\|f(v)\|_p^{p}.\\
\endaligned\end{equation}
Letting $w_\mu(\cdot)=\mu^{\frac{2}{p-2}}v\big(\mu^{\frac{p-4}{2p-4}}\cdot\big)$ and taking $r=\mu^{\frac{2}{p-2}}$ with $s=\mu^{\frac{p-4}{2p-4}}$, by \eqref{eqBB15}, there exists $\kappa>0$ such that
 \begin{equation}\label{eqBBB15}
\frac{J_\mu(w_\mu)}{\mu}\geq{1\over 2}\|\nabla
v\|_2^2+\big(\kappa-C_\varepsilon\mu^{-1} \big)\|f(v)\|_2^2-\varepsilon \|f(v)\|_{p}^{p}.
\end{equation}
For $A>0$, let $$J_{A,\varepsilon}(v)={1\over 2}\|\nabla
v\|_2^2+A\|f(v)\|_2^2-\varepsilon \|f(v)\|_{p}^{p}.$$
As in the proof of Theorem 1.1 of \cite{YangTang2019JMP}, there exists $v_{A,\varepsilon}\in H_r^1(\mathbb{R}^N)$ such that $J_{A,\varepsilon}(v_{A,\varepsilon})=c(A,\varepsilon)>0$ and  $J^\prime_{A,\varepsilon}(v_{A,\varepsilon})=0$, where  the minimax level is defined by
$$c(A,\varepsilon):=\inf\limits_{\gamma\in\Gamma_{A,\varepsilon}}\max
\limits_{t\in[0,1]}J_{A,\varepsilon}(\gamma(t)),$$
and $$\Gamma_{A,\varepsilon}=\{\gamma\in C\big([0,1],H_r^1(\mathbb{R}^N)\big):\gamma(0)=0, ~J_{A,\varepsilon}(\gamma(1))<0\}.$$

\textbf{Claim 1:} $\lim_{\varepsilon\rightarrow0}c(A,\varepsilon)=+\infty.$

Since
\begin{equation}\aligned\label{eqB115}
\|\nabla v_{A,\varepsilon}\|_2^2+A\|f(v_{A,\varepsilon})\|_2^2&\leq
\|\nabla v_{A,\varepsilon}\|_2^2+2A\int_{\mathbb{R}^N}f(v_{A,\varepsilon})f^{\prime}(v_{A,\varepsilon})v_{A,\varepsilon}\\
&=\varepsilon p\int_{\mathbb{R}^N}\big|f(v_{A,\varepsilon})\big|^{p-2}f(v_{A,\varepsilon})f^{\prime}(v_{A,\varepsilon})v_{A,\varepsilon}\\
&\leq C\varepsilon\|f(v_{A,\varepsilon}\|_2^2+C\varepsilon\|f(v_{A,\varepsilon})\|_{22^\ast}^{22^\ast}\\
&\leq C\varepsilon\big(\|\nabla v_{A,\varepsilon}\|_2^2+\|f(v_{A,\varepsilon})\|_2^2\big)
+C\varepsilon\bigg(\|\nabla v_{A,\varepsilon}\|_2^2+\|f(v_{A,\varepsilon})\|_2^2\bigg)^{\frac{2^*}{2}},
\endaligned\end{equation}
one has $\|\nabla v_{A,\varepsilon}\|_2^2+\|f(v_{A,\varepsilon})\|_2^2\geq C$ for samll $\varepsilon$. Moreover, \eqref{eqB115} leads to
$$\|\nabla v_{A,\varepsilon}\|_2^2+\|f(v_{A,\varepsilon})\|_2^2\rightarrow +\infty~~\hbox{as}~~\varepsilon\rightarrow0.$$
Consequently,
\begin{equation}\aligned\label{eqB116}
c(A,\varepsilon)&=J_{A,\varepsilon}(v_{A,\varepsilon})-\frac{2}{p}\langle J^\prime_{A,\varepsilon}(v_{A,\varepsilon}),v_{A,\varepsilon}\rangle\\
&\geq \bigg(1-\frac{2}{p}\bigg)\big(\|\nabla v_{A,\varepsilon}\|_2^2+A\|f(v_{A,\varepsilon})\|_2^2\big)\rightarrow+\infty~~\hbox{as}~~\varepsilon\rightarrow0.
\endaligned\end{equation}

\textbf{Claim 2:} $c(A,\varepsilon)$ is continuous with respect to $A$.

Without loss of generality, we may assume  $\varepsilon=1$.
For fixed $A$, we denote $J_\delta:=J_{A+\delta,1}$, so do $c_\delta$ and $\Gamma_{\delta}$.
Clearly, $c_\delta\leq c_0$ if $\delta<0$.
Suppose that $\delta_k\rightarrow 0^-$, we claim that $\lim_{k\rightarrow\infty}c_{\delta_k}=c_0$. Otherwise, up to a subsequence,
$\lim_{k\rightarrow\infty}c_{\delta_k}:=\underline{c}<c_0$, because $c(A,1)\leq c(B,1)$ if $A<B$.
Take $\{v_k\}\subset H_r^1(\mathbb{R}^N)$ with $\|v_k\|$=1 such that
$$\max\limits_{t\geq0}J_{\delta_k}(tv_k)\leq c_{\delta_k}+\frac{1}{k^2}. $$
Using the proof of  Lemma 2.5 in  \cite{YangTang2019JMP} and Proposition 3.11 in \cite{Rabinowitz1992ZAMP}, there exists $\gamma_k \in \Gamma_k$  with $\gamma_k(t)=\psi_k(t)v_k$ such that
$$\max\limits_{t\in[0,1]}J_{\delta_k}(\gamma_k(t))=\max\limits_{t\geq0}J_{\delta_k}(tv_k)\leq c_{\delta_k}+\frac{1}{k^2}.$$
According to Theorem $4.3$ of \cite{Willem1989book}, there exist $\{w_k\}\subset H_r^1(\mathbb{R}^N)$ and $\{t_k\}\subset[0,1]$ such that
\begin{equation}\aligned\label{eqBB150}
&\|w_k-\gamma_k(t_k)\|\leq\frac{1}{k},\\
&J_{\delta_k}(w_k) \in\big(c_{\delta_k}-\frac{1}{k},c_{\delta_k}+\frac{1}{k}\big),\\
&\|J^{\prime}_{\delta_k}(w_k)\|\leq \frac{1}{k}.
\endaligned\end{equation}
Let $\max\limits_{t\geq0}J_0(tv_k):=J_0(\varphi(v_k)v_k)$. Then
\begin{equation}\aligned\label{eqBB151}
c_0&\leq\max\limits_{t\geq0}J_0(tv_k)\\
&=J_0(\varphi(v_k)v_k)\\
&= J_{\delta_k}(\varphi(v_k)v_k)-\delta_k\|f[\varphi(v_k)v_k]\|_2^2\\
&\leq \max\limits_{t\in[0,1]}J_{\delta_k}(tv_k)-\delta_k\|f[\varphi(v_k)v_k]\|_2^2\\
&\leq c_{\delta_k}+\frac{1}{k^2}-\delta_k\|f[\varphi(v_k)v_k]\|_2^2.
\endaligned\end{equation}
Up to a subsequence, if  ${\varphi(v_k)}$ is bounded above,
by  $$\|f[\varphi(v_k)v_k]\|_2^2\leq \varphi^2(v_k)$$
and  \eqref{eqBB151},  a contradiction  will be  obtained.
Thus we may assume that $\varphi(v_k)\geq1$ for all $k$, by Lemma \ref{LemA1}, we have
\begin{equation*}\aligned
C_1\varphi^2(v_k)&\geq\varphi^2(v_k)\|\nabla v_k\|_2^2+2A\int_{\mathbb{R}^N}f[\varphi(v_k)v_k]f^{\prime}[\varphi(v_k)v_k]\varphi(v_k)v_k\\
&= p\int_{\mathbb{R}^N}\big|f[\varphi(v_k)v_k]\big|^{p-2}f[\varphi(v_k)v_k]f^{\prime}[\varphi(v_k)v_k]\varphi(v_k)v_k\\
&\geq C_2(\varphi(v_k))^{\frac{p}{2}}\|v_k\|_{\frac{p}{2}}^{\frac{p}{2}},
\endaligned\end{equation*}
which leads to
 $$\varphi(v_k)\leq \|v_k\|_{\frac{p}{2}}^{\frac{p}{8-2p}}.$$
If, along a subsequence, $\|v_k\|_{\frac{p}{2}}\geq C>0$, then we get upper bound for ${\varphi(v_k)}$. Otherwise,
$$v_k\rightarrow0~~\hbox{in}~~L^{\frac{p}{2}}(\mathbb{R}^N).$$
We claim that there exist $R,\delta>0$ and $\{y_k\}\subset \mathbb{R}^N$ such that
 \begin{equation*}\liminf_{n\rightarrow\infty}\int_{B_R(y_k)}
 |w_k|^2\geq\delta.\end{equation*}
If not, for any $r>0$
$$\lim_{k\rightarrow\infty}\sup_{y\in \mathbb{R}^N}\int_{B_r(y)}|w_k|^2=0.$$  By \eqref{eqBB150},
$\{w_k\}$ is bounded in $ H_r^1(\mathbb{R}^N)$. It follows from Lemma 1.21 of \cite{Willem1996book} and its proof that $v_k\rightarrow 0$ in  $L^s(\mathbb{R}^N)$, where $2\leq s<2^*$. Consequently,
 \begin{align*}
0<c_0\leq c_{\delta_k}&=\lim_{k\rightarrow\infty}\bigg(J_{\delta_k}(w_k)-\langle J^{\prime}_{\delta_k}(w_k),w_k\rangle\bigg)\\
&\leq p\int_{\mathbb{R}^N}|f(w_k)|^{p-2}f(w_k)f^{\prime}(w_k)w_k+\|f(w_k)\|_p^p\\
&\leq C\|w_k\|_{\frac{p}{2}}^{\frac{p}{2}}\rightarrow0,
 \end{align*}
a contradiction. In view of  $$\psi_k(t)=\|\psi_k(t)v_k\|\leq\|w_k-\psi_k(t)v_k\|+\|w_k\|\leq C,$$  we have for large $k$,
 \begin{align*}
\|v_k\|_2&\geq\|v_k\|_{2,B_R(y_k)}\\
&\geq {\psi_k(t)}^{-1}\big(\|w_k\|_{2,B_R(y_k)}-\|w_k-\psi_k(t)v_k\|_{2,B_R(y_k)}\big)\\
&\geq C^{-1}\big(\delta-\frac{1}{k}\big)\\
&\geq \frac{\delta}{2C},
\end{align*}
a contradiction.

Suppose that $\delta_k\rightarrow 0^+$. If, up to a subsequence,
$\lim_{k\rightarrow\infty}c_{\delta_k}:=\overline{c}>c_0$.
Take $\{v_k\}\subset H_r^1(\mathbb{R}^N)$ with $\|v_k\|$=1 such that
$$\max\limits_{t\geq0}J_0(tv_k)\leq c_0+\frac{1}{k^2}. $$
Observe that
\begin{equation}\aligned\label{eqBB151}
c_{\delta_k}&\leq\max\limits_{t\geq0}J_{\delta_k}(tv_k)\\
&=J_{\delta_k}(\varphi_{\delta_k}(v_k)v_k)\\
&= J_0(\varphi_{\delta_k}(v_k)v_k)+\delta_k\|f[\varphi_{\delta_k}(v_k)v_k]\|_2^2\\
&\leq \max\limits_{t\in[0,1]}J_0(tv_k)+\delta_k\|f[\varphi_{\delta_k}(v_k)v_k]\|_2^2\\
&\leq c_0+\frac{1}{k^2}+\delta_k\|f[\varphi_{\delta_k}(v_k)v_k]\|_2^2\\
&\leq c_0+\frac{1}{k^2}+\delta_k\varphi_{\delta_k}(v_k)^2,
\endaligned\end{equation}
where  $\max\limits_{t\geq0}J_{\delta_k}(tv_k):=J_{\delta_k}(\varphi_{\delta_k}(v_k)v_k)$.
Similarly, we can obtain a contradiction by proving that  $\{\varphi_{\delta_k}(v_k)\}$ is  bounded above.

For any $\gamma\in\Gamma_1(\mu)$ and  $\xi\in [0,1]$, let $$\eta(\xi)(\cdot)=\mu^{-\frac{2}{p-2}}\gamma(\xi)\big(\mu^{-\frac{p-4}{2p-4}}\cdot\big).$$
 In view of \eqref{eqBBB15},  we get $\eta\in\Gamma_{\kappa-C_\varepsilon\mu^{-1} ,\varepsilon}$ and
\begin{align*}
\max_{\xi\in[0,1]}J_\mu(\gamma(\xi))=\max_{\xi\in [0,1]}J_\mu({\eta(\xi)}_\mu)\geq \mu \max_{\xi\in [0,1]}J_{\kappa-C_\varepsilon\mu^{-1},\varepsilon}(\eta(\xi))\geq \mu c\big(\kappa-C_\varepsilon\mu^{-1},\varepsilon\big),
\end{align*}
 which implies that $$\frac{a_1(\mu)}{\mu}\geq c\big(\kappa-C_\varepsilon\mu^{-1},\varepsilon\big).$$  It follows claim $2$ that
$$\lim\limits_{\mu\rightarrow+\infty}\frac{a_1(\mu)}{\mu}\geq \lim\limits_{\mu\rightarrow+\infty}c\big(\kappa-C_\varepsilon\mu^{-1},\varepsilon\big)=c\big(\kappa,\varepsilon\big).$$
Since $\varepsilon>0$ is arbitrary, by  claim $1$, we have
$$\lim\limits_{\mu\rightarrow+\infty}\frac{a_1(\mu)}{\mu}=+\infty.$$
\end{proof}
\subsection{Construction of Pohozaev Mountain and Minimax level }
Following \cite{CingolaniTanaka2022CVPDE}, we introduce the Pohozaev level set
$$\Omega:=\{(\lambda,v)\in\mathbb{R}\times H_r^1(\mathbb{R}^N):P(\lambda,v)>0\}\cup\{(\lambda,0):\lambda\in\mathbb{R}\}.$$
As in the proof of \eqref{eqBbb1}, we have $P(\lambda,v)>0$  for  $v\in  \partial \Sigma_\rho$ if $\rho$ is small. Hence $\{(\lambda,0):\lambda\in\mathbb{R}\}\subset \hbox{int}(\Omega)$, and hence
$$\partial\Omega=\{(\lambda,v)\in\mathbb{R}\times H_r^1(\mathbb{R}^N):P(\lambda,v)=0,v\neq0\}$$
Combining  Theorem $1.1$ in \cite{LZL2022Nonlinearity} and a Pohozaev identity given in Lemma $3.1$ in \cite{JeanjeanNonlinearity2010}, we have $\partial\Omega\neq\emptyset$. We also remark that nonlinearity  odd is required and thus the existence of infinitely many solutions can be obtain in \cite{LZL2022Nonlinearity}. However, we can still prove the existence of  solutions if we remove the  odd  hypothesis.
\begin{lemma}\label{LemA3}The following statements hold:
\begin{itemize}
\item[$(1)$] $J_\lambda(v)\geq0$ for all $(\lambda,v)\in \Omega$.
\item[$(2)$] $J_\lambda(v)\geq a_1(\lambda)>0$ for all $(\lambda,v)\in \partial\Omega.$
\item[$(3)$] $E_m\geq B_m>-\infty$, where $$E_m:=\inf_{(\lambda,v)\in\partial\Omega}J^m(\lambda,v)~~\hbox{and}~~B_m:=\inf_{\lambda\in(-\infty,\lambda_0)}\bigg(a_1(\lambda)-{ e^\lambda  \over2}m\bigg).$$
\end{itemize}
\end{lemma}
\begin{proof}
Note that
$$J_\lambda(v)\geq J_\lambda(v)-\frac{1}{N}P(\lambda,v)=\frac{1}{N}\|\nabla v\|_2^2\geq0$$
and $(1)$ follows.

For any $(\lambda,v)\in \partial\Omega$, we claim that  $\lambda<\lambda_0.$  Otherwise, by the definition of $\lambda_0$, we have
$$G(s)-{ Ne^\lambda \over2}s^2\leq0~~\hbox{for}~~\hbox{all}~~s\in\mathbb{R}.$$
This leads to
\begin{align*}
0=P(\lambda,v)={N-2\over 2}\|\nabla v\|_2^2-N\bigg(\int_{\mathbb{R}^N}G[f(v)]-{e^\lambda  \over2}\|f(v)\|^2\bigg)\geq{N-2\over 2}\|\nabla v\|_2^2>0,
\end{align*}
a contradiction.
For any $(\lambda,v)\in \partial\Omega$, we have
\begin{align*}
0<{N-2\over 2}\|\nabla v\|_2^2=N\bigg(\int_{\mathbb{R}^N}G[f(v)]-{e^\lambda  \over2}\|f(v)\|_2^2\bigg).
\end{align*}
 Set $h(t)=\frac{d}{dt}\big(J_\lambda(v(\cdot/t))\big)$. It follows from
$$h^\prime(t)=\frac{N-2}{2}t^{(N-2)}\|\nabla v\|_2^2(1-t^2)$$
that
 $$\max_{t\geq0}J_\lambda(v(\cdot/t))=J_\lambda(v).$$
Since $J_\lambda(v(\cdot/t)\rightarrow-\infty$ as $t\rightarrow+\infty$, there exists $t_0>0$ such that $J_\lambda(v(\cdot/t_0))<0$. Let
\begin{equation*}
\gamma(t)=
\begin{cases}
v(\cdot/tt_0),~~~\hbox{if}~~0<t\leq1,\\
0,~~~~~~~~~~~\hbox{if}~~t=0.
\end{cases}
\end{equation*}
Clearly,   $\gamma\in{\Gamma}_1(\lambda)$ and
$${J_\lambda}(v)=\max_{\xi \in [0,1]}J_\lambda(\gamma(t))\geq a_1(\lambda).$$
Passing to the infimum over ${\Gamma}_1(\lambda)$, $(2)$ follows.
Hence $E_m\geq B_m$, while the fact that $B_m>-\infty$ follows from Lemma \ref{LemAA2}.
\end{proof}

Next, let us define a family of minimax values as follows.

\begin{Def}
For any $k=1$, set
$$b^m_1=\inf_{\Upsilon \in \Gamma^m_1}\max_{\xi \in [0,1]}J^m(\gamma(\xi)),$$ where
$$\Gamma^m_1=\{\Upsilon=(\Upsilon_1,\Upsilon_2)\in C([0,1],\mathbb{R}\times H_r^1(\mathbb{R}^N)):\Upsilon~\hbox{satisfies}~(\rm{\expandafter{\romannumeral1}})-(\rm{\expandafter{\romannumeral2}})~~\hbox{below}\}.$$
\begin{itemize}
\item[\rm{(\expandafter{\romannumeral1})}]   $J^m(\Upsilon(0))\leq B_m-1$.
\item[\rm{(\expandafter{\romannumeral2})}]  $\Upsilon(1)\notin \Omega$ and $J^m(\Upsilon(1))\leq B_m-1$.
\end{itemize}
For any $k \geq2$, set
$$b^m_k=\inf_{\Upsilon \in \Gamma^m_k}\max_{\xi \in D_k}J^m(\gamma(\xi)),$$ where
$$\Gamma^m_k=\{\Upsilon=(\Upsilon_1,\Upsilon_2)\in C(D_k,\mathbb{R}\times H_r^1(\mathbb{R}^N)):\Upsilon~\hbox{satisfies}~(\rm{\expandafter{\romannumeral1}})-(\rm{\expandafter{\romannumeral3}})~~\hbox{below}\}.$$
\begin{itemize}
 \item[\rm{(\expandafter{\romannumeral1})}]  $\Upsilon\in\Gamma^m_k$ to be $\mathbb{Z}_2$-equivariant, that is $\Upsilon_1(-\xi)=\Upsilon_1(\xi)$, $\Upsilon_2(-\xi)=-\Upsilon_2(\xi)$ for all $\xi\in D_k$.
\item[\rm{(\expandafter{\romannumeral2})}]   $J^m(\Upsilon(0))\leq B_m-1$.
\item[\rm{(\expandafter{\romannumeral3})}]  $\Upsilon|_{\partial D_k}\cap \Omega=\emptyset$ and $J^m(\Upsilon|_{\partial D_k})\leq B_m-1$.
\end{itemize}
\end{Def}

\begin{lemma}\label{LemA4}
  For any $k\in\mathbb{N}$, $\Gamma^m_k\neq\emptyset$. Moreover, $B_m\leq b^m_k<0$  if $m>m_k$.
\end{lemma}
\begin{proof}
Here we only prove the case for $k\geq2$, since the case $k=1$ is similar.
In order to prove $B_m\leq b^m_k$, it suffices to prove that any path in  $ \Gamma^m_k$ passes through $\partial\Omega$.  In fact, for all $\Upsilon \in \Gamma^m_k$, there exist $\xi_1 \in \hbox{int}~(D_k)$ and $\xi_2\in\hbox{int}~(D_k)$ such that $P(\Upsilon(\xi_1))>0$  and $P(\Upsilon(\xi_2))<0$. Let $$\alpha:=\sup\{t\in[0,1]:\chi(t)>0\},~~\beta:=\inf\{t\in[0,1]:\chi(t)<0\},$$ where $\chi(t)=P(\Upsilon(t\xi_1+(1-t)\xi_2)).$ Clearly, $0<\alpha\leq\beta<1$ and $\chi(\alpha)=\chi(\beta)=0$.  We divide the proof into two cases as follows.

\textbf{Case 1:} If $\alpha=\beta$, for small $\delta>0$, by the definition of $\alpha$ and $\beta$,  there exist $t_1, t_2\in (\alpha-\delta,\alpha+\delta)$ such that $\chi(t_1)>0$ and
$\chi(t_2)<0$. By the fact that $\{(\lambda,0)|\lambda\in\mathbb{R}\}\subset \hbox{int}(\Omega)$, we have $\Upsilon_2(\alpha\xi_1+(1-\alpha)\xi_2))\neq0$, and hence $\Upsilon(\alpha\xi_1+(1-\alpha)\xi_2))\in \partial\Omega$.

\textbf{Case 2:} If $\alpha<\beta$, for small $\delta>0$, there exist $t_1\in (\alpha-\delta,\alpha)$ such that $\chi(t_1)>0$ and
$\chi(t)\equiv0$ for $t\in[\alpha,\beta]$. Thus  $\Upsilon(\alpha\xi_1+(1-\alpha)\xi_2))\in \partial\Omega$.

 Fix $\lambda\in(-\infty,\lambda_0)$, for any $\gamma\in\Gamma_k(\lambda)$ and $\xi\in \partial D_k$,  we have
\begin{align*}
J_\lambda\big(\gamma(\xi)(L^{-1}\cdot)\big)&={1\over 2}L^{N-2}\|\nabla
 \gamma(\xi)\|_2^2+{ e^\lambda  \over2}L^N\|f(\gamma(\xi))\|_2^2
-L^N\int_{\mathbb{R}^N}G[f(\gamma(\xi))]\\
&\leq{1\over 2}L^{N}\|\nabla\gamma(\xi)\|_2^2+{ e^\lambda  \over2}L^N\|f(\gamma(\xi))\|_2^2
-L^N\int_{\mathbb{R}^N}G[f(\gamma(\xi))]
\\
&\leq L^N M\rightarrow -\infty,~~~~~~~~~~~~~~~~~~~~\hbox{as}~~L\rightarrow\infty,
\end{align*}
where $M=\max\limits_{\xi\in\partial D_k}J_\lambda(\gamma(\xi))<0$. Consequently, $\lim\limits_{L\rightarrow\infty}J_\lambda\big(\gamma(\xi)(L^{-1}\cdot)\big)=-\infty$ uniformly  for $\xi\in \partial D_k$.
Taking large $L>0$ and setting
\begin{equation*}
\gamma_0(\xi)(\cdot)=
\begin{cases}
 \gamma(2\xi)(\cdot)~~~~~~~~~~~~~~~~~~~~~~~~~~~~~\hbox{if}~~|\xi|\in[0,\frac{1}{2}],\\
 \gamma\big(\frac{\xi}{|\xi|}\big)\big(((2|\xi|-1)L+1)^{-1}\cdot\big)~~\hbox{if}~~|\xi|\in(\frac{1}{2},1],
\end{cases}
\end{equation*}
then $\gamma_0\in\Gamma_k(\lambda)$. Moreover,
 \begin{equation}\label{eqB7} \max_{\xi\in D_k}J_\lambda(\gamma_0(\xi))=\max_{\xi\in D_k}J_\lambda(\gamma(\xi)),~~\max_{\xi\in \partial D_k}J_\lambda(\gamma_0(\xi))\leq B_m-1. \end{equation}
Now, let us define $\Upsilon=(\Upsilon_1,\Upsilon_2)$ as follows:
\begin{equation*}
\Upsilon_1(\xi)=
\begin{cases}
 \lambda+R(1-2|\xi|)~~~\hbox{if}~~|\xi|\in[0,\frac{1}{2}],\\
 \lambda~~~~~~~~~~~~~~~~~~~~~\hbox{if}~~|\xi|\in(\frac{1}{2},1],
\end{cases}
\end{equation*}
\begin{equation*}
\Upsilon_2(\xi)=
\begin{cases}
 0~~~~~~~~~~~~~~~~~~~~~\hbox{if}~~|\xi|\in[0,\frac{1}{2}],\\
 \gamma_0\big(\frac{\xi}{|\xi|}(2|\xi|-1)\big)~~\hbox{if}~~|\xi|\in(\frac{1}{2},1].
\end{cases}
\end{equation*}
It is easy to see that
$$J^m(\Upsilon(\xi))=-\frac{e^{\lambda+R(1-2|\xi|)}}{2}m~~\hbox{for}~~|\xi|\in[0,\frac{1}{2}],$$
$$J^m(\Upsilon(\xi))=J^m(\lambda,\Upsilon_2(\xi))= J_\lambda(\Upsilon_2(\xi))-\frac{e^{\lambda}}{2}m~~\hbox{for}~~|\xi|\in(\frac{1}{2},1].$$
By choosing a large $R>0$, we have  $J^m(\Upsilon(0))\leq B_m-1$. If $\xi\in \partial D_k$,  by \eqref{eqB7},  we have for  large $L$
$$J^m(\Upsilon(\xi))=J_\lambda(\gamma_0(\xi))-\frac{e^{\lambda}}{2}m\leq B_m-1$$
and $$P(\Upsilon(\xi))=P(\lambda,\gamma_0(\xi))=P\big(\lambda,\gamma_0(\xi)(L+1)^{-1}\cdot)\big)<0.$$
 Consequently, $\Upsilon \in \Gamma_k^m$. By \eqref{eqB7}, we have
\begin{align*}
b^m_k&\leq\max_{\xi\in D_j}J^m(\Upsilon(\xi))\\
&=\max_{|\xi|\in(\frac{1}{2},1]}J^m(\Upsilon(\xi))\\
&=\max_{\xi\in D_k}J_\lambda(\gamma_0(\xi))-\frac{e^{\lambda}}{2}m\\
&=\max_{\xi\in D_k}J_\lambda(\gamma(\xi))-\frac{e^{\lambda}}{2}m.
\end{align*}
Since $\gamma$ is arbitrary, for any $\lambda \in (-\infty,\lambda_0),$ \begin{equation}\label{eq213}b^m_k\leq a_k(\lambda)-\frac{e^{\lambda}}{2}m.\end{equation}  Hence
$$\inf_{\lambda(-\infty,\lambda_0)}\frac{b^m_k}{e^{\lambda}}\leq \inf_{\lambda\in(-\infty,\lambda_0)}\frac{a_k(\lambda)}{e^{\lambda}}-\frac{m}{2}=\frac{m_k}{2}-\frac{m}{2}<0,$$
and hence $b^m_k<0$. This completes the proof of Lemma \ref{LemA4}. \end{proof}

In order to study  the multiplicity of radial  for equation $(P_{\mu,m})$, let us define another family of minimax values as follows:
\begin{Def} For $k\geq2$, we set
$$c_k^m=\inf_{A\in \Lambda^m_k}\max_{(\lambda,v)\in A}J^m(\lambda,v),$$ where
$$\Lambda_k^m=\{\Upsilon(\overline{D_{k+j}\setminus Y}):j\in\mathbb{N},\Upsilon \in\Gamma^m_{k+j},Y\subset D_{k+j}\setminus\{0\}~~\hbox{is}~~\hbox{closed},~-Y=Y,~\hbox{and}~~\hbox{genus}(Y)\leq j\}.$$
\end{Def}
 Let $P_2: \mathbb{R}\times H_r^1(\mathbb{R}^N)\rightarrow H_r^1(\mathbb{R}^N)$ be a projection  defined by
 $$P_2(\lambda,v)=v~~~
\hbox{for}~~(\lambda,v)\in \mathbb{R}\times H_r^1(\mathbb{R}^N).$$

\begin{lemma}\label{LemA6}We have the following results:
\begin{itemize}
\item[(\expandafter{\romannumeral1})] $\Lambda^m_k\neq \emptyset.$
\item[(\expandafter{\romannumeral2})]$\Lambda^m_k\subset \Lambda^m_{k+1}$, and thus $c^m_k\leq c^m_{k+1}$.
\item[(\expandafter{\romannumeral3})]$c^m_k\leq b^m_k.$
\item[(\expandafter{\romannumeral4})] Let $A\in \Lambda^m_k$ and $Z\subset\mathbb{R}\times H_r^1(\mathbb{R}^N)$ be $\mathbb{Z}_2$-invariant, closed, and such that $0\notin \overline{P_2(Z)}$ and $\hbox{genus}\overline{P_2(Z)}\leq i<n$. Then $\overline{A\setminus Z}\in\Lambda^m_{k-i}.$
\end{itemize}
\end{lemma}
\begin{proof}
The proof is  given in \cite{Tanaka2019ANS} and \cite{TanakaCingolaniNonlinearity2021}.
\end{proof}

\begin{lemma}\label{LemA12}The following statements hold:
\begin{itemize}
\item[(\expandafter{\romannumeral1})]$b_1^m$ and $c_k^m$ are  critical values of $J^m(\lambda,v)$ for all $k \geq2$.
\item[$(\expandafter{\romannumeral2})$]  If $c_k=c_{k+1}=\cdot\cdot\cdot =c_{k+j}:=c$, then $\hbox{genus}(P_2(K^m_c))\geq j+1$.
\end{itemize}
\end{lemma}
\begin{proof}
 Since the $(PSP)_c$  condition holds for $c<0$ by Lemma \ref{LemA7}, we can develop deformation theory given in Lemma \ref{LemA11}. We can also observe that the negative minimax values $b_1^m$ and $c_k^m$
are stable under the deformation. Thus (\expandafter{\romannumeral1}) follows from Lemma \ref{LemA11}, see  \cite{Myself2021JMAA} for details. Essentially,  the proof of (\expandafter{\romannumeral2}) is similar as Proposition 3.3 in \cite{Tanaka2019ANS}. Here we give some details for completeness. By Lemma \ref{LemA11}--(\expandafter{\romannumeral1}),
$P_2(K_c^m)$ is compact, symmetric with respect to $0$ and $0\notin P_2(K_c^m)$.
 By the fundamental properties of genus \cite{Rabinowitzbook}, we have
\begin{itemize}
\item[(1)] $\hbox{genus}(P_2(K_c^m))\leq \infty$,
\item[(2)] there exists $\delta>0$ small such that $\hbox{genus}(\overline{P_2(N_\delta(K_c^m))})=\hbox{genus}(P_2(K_c^m))$.
\end{itemize}
Let $\varepsilon\in (0,1)$ and $A\in\Lambda^m_{k+j}$ such that $A\subset {(J^m)}^{c-\varepsilon}$. By Lemma \ref{LemA11}--(4), one has
 $$\eta(1,\overline{A\setminus P_2(N_\delta(K_c))})\subset {(J^m)}^{c-\varepsilon}.$$
 We argue by contradiction that $\hbox{genus}(P_2(K_c^m))\leq j$.  By Lemma 2.4--(\expandafter{\romannumeral4}) in \cite{Tanaka2019ANS}, one can check that
 $$\overline{A\setminus P_2(N_\delta(K_c^m))}\in \Lambda^m_{k+j-\hbox{genus}(\overline{P_2(N_\delta(K_c^m))})}\subset\Lambda_{k}^m.$$
Through a similar argument to (\expandafter{\romannumeral1}), we have  $\eta(1,\overline{A\setminus P_2(N_\delta(K_c^m))})\in \Lambda^m_{k}$, and hence
$$J^m(\eta(1,\overline{A\setminus P_2(N_\delta(K_c^m))}))\geq c.$$
This leads to $c\leq c-\varepsilon$, a contradiction.
\end{proof}
\section{The proofs of main results}

\textbf{ The proof of Theorems \ref{Thm1.10}--\ref{Thm1.1} }
\begin{proof}
By Lemma \ref{LemCC1} and Lemma \ref{LemA12}, it suffices to prove that $b_1^m$ be least energy on
$$\{u : u \in S(m),\mu\in\mathbb{R} ~\hbox{and} ~(u,\mu)~ \hbox{solves} ~(P_{\mu,m})\}.$$
 We argue by contradiction that  there exists  $(\mu_\ast,u_\ast)$ solves $(P_{\mu,m})$ with $E(u_\ast)<b_1^m<0$, that is,  $(\mu_\ast,v_\ast)$ satisfies
$$-\Delta v +\mu f(v)f^\prime(v)=g[f(v)]f^\prime(v)$$
and $\|f(v_\ast)\|_2^2=m$, where $v_\ast:=f^{-1}(u_\ast)$.
By corollary B.4 in \cite{Willem1996book}, we have $v_\ast$ satisfies  Pohozaev identity
$$\frac{N-2}{2}\|\nabla v_\ast\|_2^2+{ N\mu_\ast}\|f(v_\ast)\|_2^2-N\int_{\mathbb{R}^N}G[f(v_\ast)]=0.$$
Since
 $$NE(u\ast)-\|\nabla v\ast\|^2+ N\mu_\ast m=\frac{N-2}{2}\|\nabla v_\ast\|_2^2+{ N\mu_\ast}\|f(v_\ast)\|_2^2-N\int_{\mathbb{R}^N}G[f(v_\ast)]=0,$$
 $\mu_\ast>0$.  Next we claim that $v\mapsto J_{\log \mu_\ast}(v)$ satisfies mountain pass geometry, that is, $\log \mu_\ast<\lambda_0$. To this aim, it is suffice to show $J_{\log \mu_\ast}(v)<0$ for some $v\in H^1_r(\mathbb{R}^N)$. Recall that  $H(s):=-\frac{\mu_\ast}{2}f^2(s)+G[f(s)]$.
If $N\geq3$, by $P(\lambda_\ast,v_\ast)=0$, we have
\begin{align*}
\int_{\mathbb{R}^N}H(v_\ast)= {N-2\over 2N}\|\nabla v_\ast\|_2^2>0.
\end{align*}
Thus, we take $v=v_\ast$ for  $N\geq3$.
If $N=2$, by $P(\lambda_\ast,v_\ast)=0$, we have $$\int_{\mathbb{R}^N}H(v_\ast)=0$$ and
$$\frac{d}{ds}\int_{\mathbb{R}^N}H(sv_\ast)\big|_{s=1}=\int_{\mathbb{R}^N}-\frac{\mu_\ast}{2}f(v_\ast)f^\prime(v_\ast)v_\ast+g[f(v_\ast)]f^\prime(v_\ast)v_\ast=\int_{\mathbb{R}^N}|\nabla v_\ast|^2>0.$$
So we can choose some $t_\ast\rightarrow1^+$ such that $\int_{\mathbb{R}^N}H(t^\ast v_\ast)>0.$  Thus, we take $v=v_\ast$ for  $N=2$.
Consequently, as $\theta\rightarrow\infty$,
\begin{align*}
J_{\log \mu_\ast}(v(\cdot/\theta))={1\over 2}\theta^{N-2}\|\nabla
v\|_2^2-\theta^N\int_{\mathbb{R}^N}H(v)\rightarrow-\infty.
\end{align*}
Fix large $\theta>0$ and let
\begin{equation*}
\gamma(t)=
\begin{cases}
v(\cdot/t\theta),~~~\hbox{if}~~0<t\leq1,\\
0,~~~~~~~~~~~\hbox{if}~~t=0.
\end{cases}
\end{equation*}
Clearly,   $\gamma\in{\Gamma}_1(\log \mu_\ast)$. It follows \eqref{eq213} that
$$E(u_\ast)<b_1^m\leq a_1(\log \mu_\ast)-\frac{\mu_\ast}{2}m\leq  \max_{t \in [0,1]}J_{\log \mu_\ast}(\gamma(t))-\frac{\mu_\ast}{2}m=J_{\log \mu_\ast}(v_\ast)-\frac{\mu_\ast}{2}m=E(u_\ast),$$
a contradiction.
\end{proof}

\textbf{ The proof of Theorems \ref{Thm1.6}}

\begin{proof}
If $m>m_1$, by  Lemma \ref{LemA12}, there exists $(\lambda_m,v_m)\in K_{b^1_m}^m$.  Let $u_m:=f(v_m)$.  Then $J^m(\lambda_m,v_m)=b^1_m$ and
$$0=\partial_{\lambda} J^m(\lambda_m,v_m)= {e^{\lambda_m}\over2}\bigg(\|u_m\|_2^2-m\bigg),$$
which implies that $\|u_m\|_2^2=m$ and
\begin{align*}
0>b^1_m=J^m(\lambda_m,v_m)&={1\over 2}\|\nabla v_m\|_2^2+{ e^{\lambda_m} \over2}\big(\|f(v_m)\|^2_2-m\big)-\int_{\mathbb{R}^N} G[f(v_m)]\\
&={1\over 2}\int_{\mathbb{R}^N}\big(1+2u_m^2\big)|\nabla
u_m|^2-\int_{\mathbb{R}^N}G(u_m).
\end{align*}
Consequently, $e(m)\leq b^1_m<0$.
Let $\{u_n\}\subset X$ be a minimizing sequence for $e(m)$, that is, $\{u_n\}\subset S(m)$ with $E(u_n)\rightarrow e_m$.  Let $u_n^\ast$ be the Steiner rearrangement of  $u_n$.  By  $(g_5)$ and $(g_7)$, and  (\expandafter{\romannumeral2})  of Proposition 2.1 in \cite{Shibata2017MZ}, we have
$$\int_{\mathbb{R}^N} G(u_n)=\int_{\mathbb{R}^N} G(u_n^\ast).$$  From standard rearrangement inequalities, we have $\|u_n\|^2_2=\|u_n^\ast\|^2_2$ and
$$\int_{\mathbb{R}^N}\big(1+2(u_n^\ast)^2\big)|\nabla u_n^\ast|^2\leq\int_{\mathbb{R}^N}\big(1+2u_n^2\big)|\nabla u_n|^2.$$  Thus,  $\{u_n\}$ can be replaced by  $\{u_n^\ast\}$ as  minimizing sequence. We still denote  $\{u_n^\ast\}$ by  $\{u_n\}$ for  simplicity.
By the Gagliardo-Nirenberg inequality, for any $2<r<22^\ast$, there exists $c(r,N)>0$ such that
\begin{equation}\aligned\label{eqaa1}\|u_n\|_r^r&\leq c(r,N)\big(\|u_n\|_2^2\big)^{\frac{3N+2-(N-2)(r-1)}{2N+4}} \big(\|\nabla u_n\|_2^2\big)^{\frac{N(r-2)}{2N+4}}\\
&\leq c(r,N)m^{\frac{3N+2-(N-2)(r-1)}{2N+4}}\bigg(\int_{\mathbb{R}^N}\big(1+2u_n^2\big)|\nabla u_n|^2\bigg)^{\frac{N(r-2)}{2N+4}}.
\endaligned\end{equation}
For any  $2<r<p$ and  $\varepsilon>0$, by  $(g_1)-(g_3)$, there exists $C_\varepsilon>0$ such that
$$|G(t)|<\varepsilon|t|^2+\varepsilon|t|^p+C_\varepsilon|t|^r,$$
 Consequently, for small $\varepsilon$,
\begin{align*}E(u_n)&={1\over 2}\int_{\mathbb{R}^N}\big(1+2u_n^2\big)|\nabla
u_n|^2-\int_{\mathbb{R}^N}G(u_n)\\
&\geq{1\over 2}\int_{\mathbb{R}^N}\big(1+2u^2\big)|\nabla
u|^2-\varepsilon\|u_n\|_2^2-C_\varepsilon\|u_n\|_r^r-\varepsilon\|u_n\|_p^p\\
&\geq\int_{\mathbb{R}^N}\big(1+2u_n^2\big)|\nabla u_n|^2-c(r,N)C_\varepsilon m^{\frac{3N+2-(N-2)(r-1)}{2N+4}} \bigg(\int_{\mathbb{R}^N}\big(1+2u_n^2\big)|\nabla u_n|^2\bigg)^{\frac{N(r-2)}{2N+4}}\\
&~~~~-c(p,N)\varepsilon m^{\frac{3N+2-(N-2)(p-1)}{2N+4}}\int_{\mathbb{R}^N}\big(1+2u_n^2\big)|\nabla u_n|^2-\varepsilon m\\
&\geq\frac{1}{4}\int_{\mathbb{R}^N}\big(1+2u_n^2\big)|\nabla u_n|^2-m-Cc(r,N) m^{\frac{3N+2-(N-2)(r-1)}{2N+4}} \bigg(\int_{\mathbb{R}^N}\big(1+2u_n^2\big)|\nabla u_n|^2\bigg)^{\frac{N(r-2)}{2N+4}}
\end{align*}
Since  $2<r<p=4+\frac{4}{N}$ and $e(m)<0$,  $\{u_n\}$ and $\{\nabla u_n^2\}$ are bounded in $H^1(\mathbb{R}^N)$ and  $L^2(\mathbb{R}^N)$, respectively. By the interpolation inequality,  we may assume that
\begin{equation*}
\label {eqD1}
\begin{cases}
u_n\rightharpoonup u~~\hbox{in}~~H_r^1(\mathbb{R}^N),\\
u_n\rightarrow u~~\hbox{in}~~L^r(\mathbb{R}^N)~~\hbox{for}~~2<r<22^\ast,\\
u_n\rightarrow u~~a.e.~~\hbox{in}~~\mathbb{R}^N.
\end{cases}
\end{equation*}
A  standard argument implies that $$\int_{\mathbb{R}^N}G(u_n)\rightarrow\int_{\mathbb{R}^N}G(u_n)$$ and
$$\liminf_{n\rightarrow\infty}\int_{\mathbb{R}^N}\big(1+2u_n^2\big)|\nabla u_n|^2\leq \int_{\mathbb{R}^N}\big(1+2u^2\big)|\nabla u|^2.$$
Then $$E(u)\leq\liminf_{n\rightarrow\infty}E(u_n)=e(m)<0.$$
From this, we have $u\neq0$. To prove Theorems \ref{Thm1.6}, it suffices to prove that $\|u\|_2^2=m$. We argue by contradiction that $0<\|u\|_2^2=c<m$, because $0<\|u\|_2^2\leq m$.
Consider the scaling $v(\cdot)=u\big(\sigma^{-\frac{1}{N}}\cdot\big)$,  where $\sigma=\frac{m}{\rho}$.  Then $\|v\|_2^2=m$ and
\begin{align*}E(v)&=\bigg(\frac{m}{\rho}\bigg)^{1-\frac{2}{N}}\bigg({1\over 2}\int_{\mathbb{R}^N}\big(1+2u^2\big)|\nabla
u|^2\bigg)-\frac{m}{\rho}\int_{\mathbb{R}^N}G(u)\\
&\leq\frac{m}{\rho}\bigg({1\over 2}\int_{\mathbb{R}^N}\big(1+2u^2\big)|\nabla
u|^2\bigg)-\frac{m}{\rho}\int_{\mathbb{R}^N}G(u)\\
&< E(u),
\end{align*}
which gives a contradiction, and the proof is completed.
\end{proof}

\textbf {Acknowledgments}

This work is supported partially by NSFC (No. 12161091).
\bibliographystyle{plain}
\bibliography{Yang-Zhao}
\end{document}